\documentclass[11pt,preprint]{imsart}
\RequirePackage[numbers]{natbib}
\RequirePackage[colorlinks,citecolor=blue,urlcolor=blue]{hyperref}
\usepackage{amsmath,amsmath,amssymb,amsfonts}
\usepackage{amsthm}
\usepackage[american]{babel}
\usepackage{graphicx}
\usepackage{flafter}
\usepackage[section]{placeins}
\usepackage{float}
\usepackage{makecell}
\usepackage{comment}
\usepackage[mathscr]{eucal}
\usepackage{bbm}
\usepackage{todonotes}
\usepackage{enumitem}
\presetkeys{todonotes}{color=red!30}{linecolor=black!50}
\startlocaldefs

\def\eps{{\varepsilon}}

\def\Bbb E{\mathbb{E}}
\def\Bbb R{\mathbb{R}}
\newtheorem{corollary}{Corollary}[section]
\usepackage{mathtools}
\makeatletter \@addtoreset{equation}{section}

\makeatother

\newtheorem{lemma}{Lemma}[section]
\newtheorem{theorem}{Theorem}[section]
\newtheorem{proposition}{Proposition}[section]

\newtheorem{remark}{Remark}[section]

\font\tencmmib=cmmib10 \skewchar\tencmmib '60
\newfam\cmmibfam
\textfont\cmmibfam=\tencmmib

\font\tenmsb=msbm10 
\def\Bbb#1{\hbox{\tenmsb#1}}

\def\lessim{\ \lower4pt\hbox{$
\buildrel{\displaystyle <}\over\sim$}\ }
\def\gessim{\ \lower4pt\hbox{$\buildrel{\displaystyle >}
\over\sim$}\ }

\def\eps{\varepsilon}

\def\go0{\to 0}

\def\leftitem#1{\item{\hbox to\parindent{\enspace#1\hfill}}}

\def\sg{\sigma}

\def\sg2{\sigma^2}

\def\__{_{\infty}}

\numberwithin{equation}{section} 

\newcommand{\1}{{\rm 1}\kern-0.24em{\rm I}}

\newtheorem{assumption}{Assumption}

\endlocaldefs

\begin{document}

\begin{frontmatter}
\title{Functional estimation in log-concave location families}
\runtitle{Functional estimation in log-concave location families}

\begin{aug}
\author{\fnms{Vladimir} \snm{Koltchinskii}\thanksref{t1}\ead[label=e1]{vlad@math.gatech.edu}} 
and 
\author{\fnms{Martin} \snm{Wahl}\thanksref{m1}\ead[label=e2]{martin.wahl@math.hu-berlin.de}}
\thankstext{t1}{Supported in part by NSF grants DMS-1810958 and DMS-2113121}
\thankstext{m1}{Supported by the Alexander von Humboldt Foundation}
\runauthor{Koltchinskii and Wahl}

\affiliation{Georgia Institute of Technology\thanksmark{m1}}

\address{School of Mathematics\\
Georgia Institute of Technology\\
Atlanta, GA 30332-0160, USA\\
\printead{e1}\\
and
\\
Department of Mathematics\\
Humboldt Universit\"at zu Berlin\\
Unter den Linden 6, 10099, Berlin, Germany\\
\printead{e2}
}
\end{aug}
\vspace{0.2cm}
{\small \today}
\vspace{0.2cm}

\begin{abstract}
Let $\{P_{\theta}:\theta \in {\mathbb R}^d\}$ be a log-concave location family with $P_{\theta}(dx)=e^{-V(x-\theta)}dx,$
where $V:{\mathbb R}^d\mapsto {\mathbb R}$ is a known convex function and let $X_1,\dots, X_n$ be i.i.d. r.v. sampled from distribution $P_{\theta}$ with an unknown location parameter $\theta.$ The goal is to estimate the value $f(\theta)$ of a smooth functional 
$f:{\mathbb R}^d\mapsto {\mathbb R}$ based on observations $X_1,\dots, X_n.$ 
In the case when $V$ is sufficiently smooth and  $f$ is a functional from a ball in a H\"older space $C^s,$
we develop estimators of $f(\theta)$ with minimax optimal error rates measured by the $L_2({\mathbb P}_{\theta})$-distance as well as 
by more general Orlicz norm distances. Moreover, we show that if $d\leq n^{\alpha}$ and $s>\frac{1}{1-\alpha},$ then the resulting estimators are  asymptotically efficient in H\'ajek-LeCam sense with the convergence rate $\sqrt{n}.$ This generalizes earlier results on estimation of smooth functionals in Gaussian shift models. The estimators have the form $f_k(\hat \theta),$ where $\hat \theta$ is the maximum likelihood estimator and 
$f_k: {\mathbb R}^d\mapsto {\mathbb R}$ (with $k$ depending on $s$) are functionals defined in terms of $f$ and designed to provide a higher order bias reduction in functional estimation problem. The method of bias reduction is based on iterative parametric bootstrap and 
it has been successfully used before in the case of Gaussian models. 
\end{abstract}



\end{frontmatter}

\section{Introduction}
\label{intro}

Let $P$ be a log-concave probability distribution in ${\mathbb R}^d$ with density $e^{-V},$ $V:{\mathbb R}^d\mapsto {\mathbb R}$ being a convex function and let $P_{\theta}$, $\theta\in {\mathbb R}^d$
be a location family generated by $P:$ $P_{\theta}(dx)=p_{\theta}(x)dx=e^{-V(x-\theta)}dx$, $\theta \in {\mathbb R}^d.$ In other words, a random variable $X\sim P_{\theta}$ could be represented as $X=\theta+\xi,$
where $\theta \in {\mathbb R}^d$ is a location parameter and $\xi\sim P$ is a random noise with log-concave distribution. Without loss of generality, one can assume that ${\mathbb E}\xi=0$
(otherwise, one can replace function $V(\cdot)$ by $V(\cdot + {\mathbb E} \xi)$).
We will also assume that function $V$ is known and $\theta$ is an unknown parameter of the model to be estimated based on i.i.d. observations $X_1,\dots, X_n$ of $X.$
We will refer to this statistical model as a {\it log-concave location family}. Our main goal is to study the estimation of $f(\theta)$ for a given smooth functional $f:{\mathbb R}^d\mapsto {\mathbb R}.$
A natural estimator of location parameter is the maximum likelihood estimator (MLE) defined as 
\begin{align*}
\hat \theta := \operatorname{argmax}_{\theta \in {\mathbb R}^d}\limits \prod_{j=1}^n p_{\theta}(X_j)=\operatorname{argmin}_{\theta\in {\mathbb R}^d}\limits \frac{1}{n}\sum_{j=1}^n V(X_j-\theta).
\end{align*}
Note that, by Lemma 2.2.1 in \cite{BGVV14} for a log-concave density $e^{-V}, V: {\mathbb R}^d\mapsto {\mathbb R}$ there exist constants $A,B>0$ such that $e^{-V(x)}\leq Ae^{-B\|x\|}$ for all $x\in\mathbb{R}^d,$ implying that $V(x)\to \infty$ as $\|x\|\to \infty.$
It is easy to conclude from this fact that MLE does exist. 
Moreover, it is unique if $V$ is strictly convex (this condition is assumed in what follows). 
In addition, MLE $\hat \theta$ is an equivariant estimator with respect to the translation group in ${\mathbb R}^d:$ 
\begin{align*}
\hat \theta (X_1+u,\dots, X_n+u) = \hat \theta (X_1,\dots, X_n)+ u, u\in {\mathbb R}^d.
\end{align*}
Also note that 
\begin{align*}
{\mathbb E}_{\theta} V(X-\theta')-{\mathbb E}_{\theta}V(X-\theta)=K(P_{\theta}\|P_{\theta'}),
\end{align*}
where $K(P_{\theta}\|P_{\theta'})$ is the Kullback-Leibler divergence between $P_{\theta}$ and $P_{\theta'},$
implying that $\theta$ is the unique minimal point 
of $\theta'\mapsto {\mathbb E}_{\theta} V(X-\theta').$
The uniqueness follows from the identifiability of parameter $\theta:$ if $\theta$ were not identifiable, we would have 
$V(x)=V(x+h), x\in {\mathbb R}^d$ for some $h\neq 0,$ which would contradict the assumption that $V(x)\to \infty$ as $\|x\|\to \infty.$ Assuming some regularity (including differentiability of $V$) one can also argue as follows : using strict convexity and a necessary condition of extremum ${\mathbb E}_{\theta}V'(X-\theta)=0,$ we can conclude that $\mathbb{E}_\theta (V(X-\theta')-V(X-\theta))>\mathbb{E}_\theta\langle V'(X-\theta),\theta-\theta'\rangle=0$ for every $\theta'\neq \theta$.

Moreover, for differentiable $V,$ the score function of location family is $\frac{\partial}{\partial \theta}\log p_{\theta}(X)= -V'(X-\theta),$
and, under some regularity,  
${\mathbb E}_{\theta}V'(X-\theta)=0.$
In addition, the Fisher information matrix of such a log-concave location family is well defined, does not depend on $\theta$ and is given by
\begin{align*} 
{\mathcal I}
&= {\mathbb E}_{\theta} \frac{\partial}{\partial \theta}\log p_{\theta}(X)\otimes \frac{\partial}{\partial \theta}\log p_{\theta}(X)= {\mathbb E}_{\theta} V'(X-\theta)\otimes V'(X-\theta)
\\
&
=
{\mathbb E} V'(\xi)\otimes V'(\xi) = \int_{{\mathbb R}^d} V'(x)\otimes V'(x) e^{-V(x)}dx
\end{align*}
(provided that the integral in the right hand side exists). 
Under further regularity, for twice differentiable $V,$ we also have (via integration by parts)
\begin{align*}
{\mathcal I} = {\mathbb E}V''(\xi)= \int_{{\mathbb R}^d} V''(x) e^{-V(x)}dx.
\end{align*}
Finally, if Fisher information ${\mathcal I}$ is non-singular, then, for a fixed $d$ and $n\to\infty,$ MLE $\hat \theta$ is an asymptotically normal estimator of $\theta$ with limit covariance ${\mathcal I}^{-1}:$
\begin{align*}
\sqrt{n}(\hat \theta-\theta)\overset{d}\to N(0;{\mathcal I}^{-1})\ {\rm as}\ n\to\infty.
\end{align*}
Assumption \ref{Assump_main} below suffices for all the above 
properties to hold.


It seems natural to estimate $f(\theta)$ by the plug-in estimator 
$f(\hat \theta),$ where $\hat \theta$ is the MLE. Such an approach yields asymptotically efficient estimators for regular statistical models in the case of fixed dimension $d$ and $n\to\infty.$ In particular, for our location family, we have by a standard application 
of the delta method that 
\begin{align*}
\sqrt{n}(f(\hat \theta)-f(\theta)) \overset{d}\to N(0; \sigma^2_f(\theta)),
\end{align*}
where $\sigma^2_f(\theta):=\langle {\mathcal I}^{-1} f'(\theta), f'(\theta)\rangle.$
However, it is well known that plug-in estimators are sub-optimal in high-dimensional problems mainly due to their large bias
and, often, non-trivial bias reduction methods are needed to achieve an optimal error rate. This has been one of the difficulties 
in the problem of estimation of functionals of parameters of high-dimensional and infinite-dimensional models for a number of 
years \cite{Levit_1, Levit_2, Ibragimov, Bickel_Ritov, Ibragimov_Khasm_Nemirov, Nemirovski_1990, Birge, Nemirovski}.

One approach to this problem is based 
on replacing $f$ by another function $g$ for which the bias of estimator $g(\hat \theta)$ is small. To find such a function $g,$
one has to solve approximately the ``bias equation" ${\mathbb E}_{\theta} g(\hat \theta)=f(\theta), \theta \in {\mathbb R}^d.$ 
This equation can be written as ${\mathcal T} g=f,$ where 
\begin{align*}
({\mathcal T} g)(\theta) := {\mathbb E}_{\theta} g(\hat \theta)= \int_{{\mathbb R}^d} g(u) P(\theta;du), \theta \in {\mathbb R}^d
\end{align*}
and $P(\theta;A), \theta\in {\mathbb R}^d, A\subset {\mathbb R}^d$ is a Markov kernel on ${\mathbb R}^d$
(or, more generally, on the parameter space $\Theta$ of statistical model), providing the distribution of estimator $\hat \theta.$ 
Denoting ${\mathcal B}:= {\mathcal T}-{\mathcal I},$ where ${\mathcal I}$ is the identity operator in the space of bounded functions on ${\mathbb R^d}$ (not to be confused with the Fisher information also denoted by ${\mathcal I}$), and assuming that $\hat \theta$
is close to $\theta$ and, as a consequence, operator ${\mathcal B}$ is ``small", one can view ${\mathcal T}={\mathcal I}+ {\mathcal B}$ as a small perturbation of identity. 
In such cases, one can try to solve the equation ${\mathcal T} g=f$ in terms of Neumann series $g= ({\mathcal I}-{\mathcal B}+{\mathcal B}^2-\dots) f.$ In what follows, we denote by 
\begin{align*}
f_k(\theta) := \sum_{j=0}^k (-1)^j ({\mathcal B}^j f)(\theta), \theta \in {\mathbb R}^d
\end{align*}
the partial sum of this series and we will use $f_k(\hat \theta)$ (for a suitable choice of $k$ depending on smoothness of functional $f$) as an estimator of $f(\theta).$ It is easy to see that its bias is 
\begin{align*}
{\mathbb E}_{\theta} f_k(\hat \theta)-f(\theta)&=(\mathcal{B}f_k)(\theta)+f_k(\theta)-f(\theta)
\\
&
= (-1)^k ({\mathcal B}^{k+1} f)(\theta), \theta \in {\mathbb R}^d.
\end{align*}
If ${\mathcal B}$ is ``small" and $k$ is sufficiently large, one can hope to achieve a bias reduction through estimator $f_k(\theta).$
Another way to explain this approach is in terms of iterative bias reduction: since the bias of plug-in estimator $f(\hat \theta)$ is equal to $({\mathcal B} f)(\theta),$ one can estimate the bias by $({\mathcal B} f)(\hat \theta)$ and the first order bias reduction yields 
the estimator $f_1(\hat \theta)= f(\hat \theta)- ({\mathcal B} f)(\hat \theta).$ Its bias is equal to $-({\mathcal B}^2 f)(\theta)$
and the second order bias reduction yields the estimator $f_2(\hat \theta)=f(\hat \theta)- ({\mathcal B} f)(\hat \theta)+ ({\mathcal B}^2 f)(\hat \theta),$ etc. This is close to the idea of iterative bootstrap bias reduction \cite{Hall_1, Hall, Jiao}. 

Let $\{\hat \theta^{(k)}: k\geq 0\}$ be the Markov chain with $\hat \theta^{(0)}=\theta$ and with transition probability kernel $P(\theta;A), \theta \in {\mathbb R}^d, A\subset {\mathbb R}^d.$ This chain can be viewed as an output of iterative application of parametric bootstrap to estimator $\hat \theta$ in the model $X_1,\dots, X_n$ i.i.d. $\sim P_{\theta}, \theta \in {\mathbb R}^d:$ at the first iteration, the data is sampled from the distribution 
with parameter $\hat \theta^{(0)}=\theta$ and estimator $\hat \theta^{(1)}=\hat \theta$ is computed; at the second iteration, the data is sampled 
from the distribution $P_{\hat \theta}$ (conditionally on the value of $\hat \theta$) and bootstrap estimator $\hat \theta^{(2)}$ is computed,
and so on. We will call $\{\hat \theta^{(k)}:k\geq 0\}$ the bootstrap chain of estimator $\hat \theta.$
Clearly, $({\mathcal T}^k f)(\theta)= {\mathbb E}_{\theta} f(\hat \theta^{(k)})$ and, by Newton's binomial formula, we also have 
\begin{align}
\nonumber
({\mathcal B}^k f)(\theta)& = (({\mathcal T}-{\mathcal I})^{k} f)(\theta) = \sum_{j=0}^{k} (-1)^{k-j}\binom{k}{j} ({\mathcal T}^{j} f)(\theta)    
\\
&
\label{representation_B^k}
= {\mathbb E}_{\theta}\sum_{j=0}^{k} (-1)^{k-j}\binom{k}{j} f(\hat \theta^{(j)}).
\end{align}
This means that $({\mathcal B}^k f)(\theta)$ is the expectation of the $k$-th order difference of function $f$
along the bootstrap chain $\{\hat \theta^{(j)}:j\geq 0\}.$ In the case when $\|\hat \theta -\theta\|\lesssim \sqrt{\frac{d}{n}}$
with a high probability, the same bound also holds for the increments $\hat \theta^{(j+1)}-\hat \theta^{(j)}$ (conditionally on $\hat \theta^{(j)}$). 
If functional $f$ is $k$ times differentiable and $d$ is small comparing with $n,$ one could therefore expect that $({\mathcal B}^k f)(\theta) \lesssim 
\bigl(\frac{d}{n}\bigr)^{k/2}$
(based on the analogy with 
the behavior of $k$-th order differences of $k$ times differentiable functions in the real line). 
The justification of this heuristic for general parametric models could be rather involved (see \cite{Koltchinskii_2017, Koltchinskii_2018, Koltchinskii_Zhilova_19, Koltchinskii_2020}), but it will be shown below that it is much simpler in the case of equivariant 
estimators $\hat \theta$ (such as the MLE) (see also \cite{Koltchinskii_Zhilova, Koltchinskii_Zhilova_2020}).

Inserting representation \eqref{representation_B^k} into the definition of $f_k$ and  using a simple combinatorial identity, we obtain the following useful representation 
of function $f_k(\theta):$
\begin{align}\label{representation_f_k_alternative}
    f_k(\theta)=\mathbb{E}_\theta\sum_{j=0}^k(-1)^{j}\binom{k+1}{j+1}f(\hat\theta^{(j)}).
\end{align}

The following notations will be used throughout the paper (and some of them have been already used). For two variables $A,B\geq 0,$
$A\lesssim B$ means that there exists an absolute constant $C>0$ such that $A\leq CB.$ The notation $A\gtrsim B$ means that 
$B\lesssim A$ and $A\asymp B$ means that $A\lesssim B$ and $B\lesssim A.$ If the constants in the relationships $\lesssim, \gtrsim, \asymp$
depend on some parameter(s), say, on $\gamma,$ this parameter will be used as a subscript of the relationship, say, $A\lesssim_{\gamma} B.$ 
Given two square matrices $A$ and $B,$ $A\preceq B$ means that $B-A$ is positively semi-definite and $A\succeq B$ means that $B\preceq A.$
The norm notation $\|\cdot\|$ (without further subscripts or superscripts) will be used by default in certain spaces. For instance, it will always 
denote the canonical Euclidean norm of ${\mathbb R}^d,$ the operator norm of matrices (linear transformations) and the operator norm 
of multilinear forms. In some other cases, in particular for functional spaces $L_{\infty}, C^s,$ etc, the corresponding subscripts will be used.

\section{Main results}

Recall that, for a convex non-decreasing function $\psi : {\mathbb R}_+\mapsto {\mathbb R}_+$ with $\psi(0)=0,$ 
the Orlicz $\psi$-norm 
of a r.v. $\eta$ is defined as 
\begin{align*}
\|\eta\|_{\psi} := \inf\Bigl\{c\geq 0: {\mathbb E}\psi\Bigl(\frac{|\eta|}{c}\Bigr)\leq 1\Bigr\}.
\end{align*}
The Banach space of all r.v. on a probability space $(\Omega, \Sigma, {\mathbb P})$ with finite $\psi$-norm is denoted by $L_{\psi}({\mathbb P}).$
If $\psi(u)=u^p, u\geq 0, p\geq 1,$ then the $\psi$-norm coincides with the $L_p$-norm. Another important 
choice is $\psi_{\alpha}(u)=e^{u^{\alpha}}-1, u\geq 0, \alpha\geq 1.$ In particular, for $\alpha=1,$ $L_{\psi_1}$ is 
the space of sub-exponential r.v. and, for $\alpha=2,$ $L_{\psi_2}$ is the space of sub-gaussian r.v.
It is also well known that the $\psi_{\alpha}$-norm is equivalent to the following norm defined in terms of moments 
(or the $L_p$-norms):
\begin{align}
\label{psi_alpha}
\|\eta\|_{\psi_{\alpha}}\asymp \sup_{p\geq 1} p^{-1/\alpha} {\mathbb E}^{1/p}|\eta|^p, \alpha\geq 1.
\end{align}
Note that the right hand side defines a norm for $0<\alpha<1,$ too, whereas the left hand side is not a norm in this case since function 
$\psi_{\alpha}$ is not convex for $0<\alpha<1.$ Relationship \eqref{psi_alpha} still holds for $0<\alpha<1,$ but with constants depending 
on $\alpha$ as $\alpha$ approaches $0.$
With a slight abuse of notations, we will define $\|\eta\|_{\psi_{\alpha}}$ by the right hand side of \eqref{psi_alpha} for all $\alpha>0.$

We will use the following definition of H\"older $C^{s}$-norms of functions $f: {\mathbb R}^d\mapsto {\mathbb R}.$ For $j\geq 0,$ $f^{(j)}$ denotes the $j$-th Fr\' echet  
derivative of $f.$ For $x\in {\mathbb R}^d,$ $f^{(j)}(x)$ is a $j$-linear form on ${\mathbb R}^d$ and the space of such forms will be equipped with the operator norm.
Clearly, $f^{(0)}=f$ and $f^{(1)}=f'$ coincides with the gradient $\nabla f.$
If $f$ is $l$ times differentiable and $s=l+\rho,$ $\rho\in (0,1],$ define
\begin{align*}
\|f\|_{C^s} := \max_{0\leq j\leq l} \sup_{x\in {\mathbb R}^d}\|f^{(j)}(x)\| \vee \sup_{x,y\in {\mathbb R}^d, x\neq y}\frac{\|f^{(l)}(x)-f^{(l)}(y)\|}{\|x-y\|^{\rho}}.
\end{align*}
We will also frequently use $L_{\infty}$ and Lipschitz norms of functions and their derivatives. For instance, 
$\|f^{(j)}\|_{L_{\infty}}= \sup_{x\in {\mathbb R}^d}\|f^{(j)}(x)\|$ and $\|f^{(j)}\|_{\rm Lip}= 
\sup_{x,x'\in {\mathbb R}^d, x\neq x'}\frac{\|f^{(j)}(x)-f^{(j)}(x')\|}{\|x-x'\|}.$

In what follows, we will use some facts related to isoperimetry and concentration 
properties 
of log-concave measures. Given a Borel probability measure $\mu$ on ${\mathbb R}^d,$ let \begin{align*}
\mu^{+}(A) := \liminf_{\eps\to 0} \frac{\mu(A_\eps)-\mu(A)}{\eps},\ A\in 
{\mathcal B}({\mathbb R}^d),
\end{align*}
where $A_{\eps}$ denotes the $\eps$-neighborhood of $A$ and ${\mathcal B}({\mathbb R}^d)$
is the Borel $\sigma$-algebra in ${\mathbb R}^d.$
The so called Cheeger isoperimetric constant of $\mu$ is defined as 
\begin{align*}
I_C(\mu):= \inf_{A\in {\mathcal B}({\mathbb R}^d)}\frac{\mu^{+}(A)}{\mu(A)\wedge (1-\mu(A))}.    
\end{align*}
According to the well known Kannan-Lov\`asz-Simonovits (KLS) conjecture, for a log-concave probability measure 
$\mu(dx)=e^{-V(x)}dx$ on ${\mathbb R}^d$ with covariance operator $\Sigma,$ 
$I_C(\mu)\gtrsim \|\Sigma\|^{-1/2}$
with a dimension-free constant. 
This conjecture remains open, but the following deep recent result by Chen \cite{Yuansi_Chen}
provides a lower bound on $I_C(\mu)$ that is almost dimension-free.

\begin{theorem}
\label{Chen}
There exists a constant $b>0$ such that, for all $d\geq 3$ and for all log-concave distributions $\mu(dx)=e^{-V(x)}dx$ in ${\mathbb R}^d$ with covariance $\Sigma,$
\begin{align*}
I_C(\mu)\geq \|\Sigma\|^{-1/2} d^{-b (\frac{\log\log d}{\log d})^{1/2}}.
\end{align*}
\end{theorem}

Isoperimetric constants $I_C(\mu)$ are known to be closely related to important functional inequalities, 
in particular, to Poincar\'e inequality and its generalizations
(see, e.g., \cite{Bobkov_Houdre, Milman}). 
It is said that Poincar\'e inequality holds for a r.v. $\xi$ in ${\mathbb R}^d$ iff, for some constant $C>0$ and for all locally Lipschitz functions $g:{\mathbb R}^d\mapsto {\mathbb R}$
(which, by Rademacher theorem, are differentiable almost everywhere), 
\begin{align*} 
{\rm Var} (g(\xi)) \leq C {\mathbb E}\|\nabla g(\xi)\|^2.
\end{align*}
The smallest value $c(\xi)$ of constant $C$ in the above inequality is called the Poincar\'e constant of $\xi$ (clearly, it depends only on the distribution of $\xi$). The following property of Poincar\'e constant will be frequently used: if r.v. $\xi = (\xi_1,\dots, \xi_n)$ has independent 
components (with $\xi_j$ being a r.v. in ${\mathbb R}^{d_j}$), then $c(\xi)=\max_{1\leq j\leq n} c(\xi_j)$ (see \cite{Ledoux}, Corollary 5.7). 

If now $\xi \sim \mu$ in ${\mathbb R}^d,$ then the following Cheeger's inequality holds (see, e.g., \cite{Milman}, Theorem 1.1):
\begin{align*}
c(\xi)\leq \frac{4}{I_C^2(\mu)}.
\end{align*}
Moreover, the following $L_p$-version of Poincar\'e inequality holds for all $p\geq 1$
and for all locally Lipschitz functions $g:{\mathbb R}^d \mapsto {\mathbb R}$ (see \cite{Bobkov_Houdre}, Theorem 3.1) 
\begin{align}
\label{L_p_Poinc}
\|g(\xi)-{\mathbb E}g(\xi)\|_{L_p} \lesssim \frac{p}{I_C(\mu)}\|\|\nabla g(\xi)\|\|_{L_p}.   \end{align}

\begin{remark}
\normalfont
\label{KLS}
Note that, if $\xi\sim \mu$ and $\mu(dx)=e^{-V(x)}dx$ is log-concave, then (see \cite{Milman}, Theorem 1.5)
\begin{align*}
c(\xi)\asymp \frac{1}{I_C^2(\mu)}
\end{align*}
and, by Theorem \ref{Chen}, we have 
\begin{align}
\label{KLS_bd}
c(\xi)\leq \|\Sigma\| \ d^{2b (\frac{\log\log d}{\log d})^{1/2}}.
\end{align}

In what follows, we denote the Poincar\'e constant $c(\xi)$ of r.v. $\xi\sim \mu$ 
with log-concave distribution $\mu(dx)=e^{-V(x)}dx$ by $c(V).$ Bound \eqref{KLS_bd}
implies that $c(V)\lesssim_{\epsilon} \|\Sigma\| d^{\epsilon}$ for all $\epsilon>0.$
Also, with this notation, we can rewrite the $L_p$-version \eqref{L_p_Poinc} of Poincar\'e inequality as follows:
\begin{align}
\label{main_conc_log_concave}
\|g(\xi)-{\mathbb E} g(\xi)\|_{L_p} \lesssim \sqrt{c(V)} p\ \|\|\nabla g(\xi)\|\|_{L_p}.
\end{align}
This concentration bound will be our main tool in Section \ref{Sec:conc}. It will be convenient for our purposes to express it in terms of local Lipschitz 
constants of $g$ defined as follows:
\begin{align*}
(L g)(x):= \inf_{U\ni x} \sup_{x',x''\in U, x'\neq x''} \frac{|g(x')-g(x'')|}{\|x'-x''\|}, x\in {\mathbb R}^d
\end{align*}
where the infimum is taken over all the balls $U$ centered at $x.$ Similar definition could be also used for vector valued 
functions $g.$ Clearly, $\|\nabla g(x)\|\leq (Lg)(x), x\in {\mathbb R}^d$ and \eqref{main_conc_log_concave} implies 
that 
\begin{align}
\label{main_conc_log_concave_A}
\|g(\xi)-{\mathbb E} g(\xi)\|_{L_p} \lesssim \sqrt{c(V)} p\ \|(Lg)(\xi)\|_{L_p}.
\end{align}
\end{remark}

The following assumptions on $V$ will be used throughout the paper.

\begin{assumption}
\label{Assump_main}
\normalfont
Suppose that 
\begin{enumerate}[label=(\roman*)]
\item $V$ is strictly convex and twice continuously differentiable such that,
for some constants $M,L>0,$ $\|V''\|_{L_{\infty}}\leq M$ and $\|V''\|_{\rm Lip}\leq L.$
\item For some constant $m>0,$
$
{\mathcal I}\succeq m I_d.
$
\end{enumerate}
\end{assumption}
Under Assumption \ref{Assump_main}, we have ${\mathcal I}=\mathbb{E}V''(\xi)\preceq M I_d$ and thus $m\leq M$. 

\begin{remark}
\normalfont
Obviously, Assumption \ref{Assump_main} holds in the Gaussian case, when 
$V(x)=c_1 + c_2\|x\|^2, x\in {\mathbb R}^d.$ In this case, 
$V''(x)= m I_d, x\in {\mathbb R}^d$ for $m=2c_2>0,$
which is much stronger than Assumption \ref{Assump_main}, (ii).
Assumption \ref{Assump_main} also holds, for instance, for  $V(x)=\varphi(\|x\|^2), x\in {\mathbb R}^d,$ where $\varphi$ is a $C^{\infty}$ function 
in ${\mathbb R}$ such that $\varphi''$ is supported in $[0,1],$ 
$\varphi''(t)\geq 0, t\in {\mathbb R}$ and $\varphi'(0)=0.$ 
Of course, in this case, the condition 
$V''(x)\succeq m I_d$ does not hold
uniformly in $x$ for any positive $m,$ but Assumption \ref{Assump_main}, (ii) holds.
If $V(x)=c_1+ \|x\|^{2p}, x\in {\mathbb R}^d$ for some $p\geq 1/2,$
it is easy to check that Assumption \ref{Assump_main} holds only 
for $p=1.$
\end{remark}

We are now ready to state our main result.

\begin{theorem}
\label{main_theorem}
Suppose Assumption \ref{Assump_main} holds and $d\leq \gamma n,$ where 
\begin{align*}
\gamma := c\Bigl(\frac{m}{M}\wedge \frac{m^2}{L \sqrt{M}}\Bigr)^2
\end{align*}
with a small enough constant $c>0.$
Let $f\in C^s$ for some $s=k+1+\rho,$ $k\geq 0,$ $\rho\in (0,1].$
Then 
\begin{align*}
&
\sup_{\theta \in {\mathbb R}^d}\Bigl\|f_k(\hat \theta) - f(\theta)- n^{-1}\sum_{j=1}^n \langle V'(\xi_j), {\mathcal I}^{-1} f'(\theta)\rangle\Bigr\|_{L_{\psi_{2/3}}({\mathbb P}_{\theta})}
\\
&
\lesssim_{L,M,m,s} \|f\|_{C^s} \Bigl[\sqrt{\frac{c(V)}{n}}\Bigl(\frac{d}{n}\Bigr)^{\rho/2} + \Bigl(\sqrt{\frac{d}{n}}\Bigr)^s\Bigr].
\end{align*}
\end{theorem}

Theorem \ref{main_theorem} shows that $f_k(\hat \theta) - f(\theta)$ can be approximated by a normalized sum of i.i.d.
mean zero r.v. $n^{-1}\sum_{j=1}^n \langle V'(\xi_j), {\mathcal I}^{-1} f'(\theta)\rangle.$ Moreover, the error of this approximation 
is of the order $o(n^{-1/2})$ provided that $d=o(n^{\alpha})$ for some $\alpha \in (0,1)$ satisfying $s>\frac{1}{1-\alpha}$ and that $\|\Sigma\|$ is bounded by a constant. 
This follows from the fact that $c(V)\lesssim_{\epsilon} d^{\epsilon}\|\Sigma\|$ for an arbitrarily small $\epsilon >0$ (see Remark \ref{KLS}).
In addition, by Lemma \ref{bd_on_V'} below,
r.v. $\langle V'(\xi), {\mathcal I}^{-1} f'(\theta)\rangle$ is subgaussian with 
\begin{align}\label{eq_V_'_subgaussian}
\|\langle V'(\xi), {\mathcal I}^{-1} f'(\theta)\rangle\|_{\psi_2} \lesssim \sqrt{M} \|{\mathcal I}^{-1} f'(\theta)\|
\lesssim \frac{\sqrt{M}}{m} \|f'(\theta)\|.
\end{align}
As a result, we can obtain the following simple, but important corollaries of Theorem \ref{main_theorem}. 
Recall that $\sigma^2_f(\theta)=\langle {\mathcal I}^{-1} f'(\theta), f'(\theta)\rangle.$

\begin{corollary}
\label{cor_L_2}
Under the conditions of Theorem \ref{main_theorem},
\begin{align*}
\sup_{\theta\in {\mathbb R}^d}\Bigl|\|f_k(\hat \theta)-f(\theta)\|_{L_2({\mathbb P}_{\theta})}
-\frac{\sigma_f(\theta)}{\sqrt{n}}\Bigr|
\lesssim_{L,M,m,s} \|f\|_{C^s} \Bigl[\sqrt{\frac{c(V)}{n}}\Bigl(\frac{d}{n}\Bigr)^{\rho/2} + \Bigl(\sqrt{\frac{d}{n}}\Bigr)^s\Bigr].
\end{align*}
\end{corollary}

This corollary immediately follows from the bound of Theorem \ref{main_theorem} and the fact that the $L_2$-norm is dominated by the  $\psi_{2/3}$-norm. 
It implies the second claim of the following proposition.

\begin{proposition}
\label{prop_min_max_upper}
Let $f\in C^s$ for some $s>0.$ \\
1. For $s\in (0,1],$
\begin{align*}
\sup_{\|f\|_{C^s}\leq 1} \sup_{\theta\in {\mathbb R}^d} \|f(\hat \theta)-f(\theta)\|_{L_2({\mathbb P}_{\theta})}
\lesssim_{L,M,m,s} \Bigl(\sqrt{\frac{d}{n}}\Bigr)^s \wedge 1.
\end{align*}
\\
2. For $s=k+1+\rho$ for some $k\geq 0$ and $\rho \in (0,1],$
suppose Assumption \ref{Assump_main} holds and also $\|\Sigma\|\lesssim 1.$ 
If $d\lesssim n^{\alpha}$ for some $\alpha\in (0,1),$ then
\begin{align}
\label{min_max_upper}
\sup_{\|f\|_{C^s}\leq 1} \sup_{\theta\in {\mathbb R}^d} \|f_k(\hat \theta)-f(\theta)\|_{L_2({\mathbb P}_{\theta})}
\lesssim_{L,M,m,s} \Bigl(\frac{1}{\sqrt{n}} \vee \Bigl(\sqrt{\frac{d}{n}}\Bigr)^s\Bigr) \wedge 1.
\end{align}
\end{proposition}

Combining bound \eqref{min_max_upper} with the following result shows some form of minimax optimality of
estimator $f_k(\hat \theta).$

\begin{proposition} 
\label{min_max_lower}
Suppose Assumption \ref{Assump_main} holds.
Then, for all $s>0,$
\begin{align*}
\sup_{\|f\|_{C^s}\leq 1}\inf_{\hat T_n}\sup_{\|\theta\|\leq 1}\|\hat T_n-f(\theta)\|_{L_2({\mathbb P}_{\theta})}\gtrsim_{m,M} \Big(\frac{1}{\sqrt{n}}&\vee\Big(\sqrt{\frac{d}{n}}\Big)^s\Big)\wedge 1,
\end{align*}
where the infimum is taken over all estimators $\hat T_n=\hat T_n(X_1,\dots,X_n).$
\end{proposition}

The proof of this result is similar to the proof of Theorem 2.2 in \cite{Koltchinskii_Zhilova_19} in the Gaussian case. 
Some further comments will be provided in Section \ref{Sec:lower_bounds}.

Corollary \ref{cor_L_2} also implies that, for all $\theta\in {\mathbb R}^d,$ 
\begin{align*}
\|f_k(\hat \theta)-f(\theta)\|_{L_2({\mathbb P}_{\theta})} \leq \frac{\sigma_f(\theta)}{\sqrt{n}}
+ C\|f\|_{C^s} \Bigl[\sqrt{\frac{c(V)}{n}}\Bigl(\frac{d}{n}\Bigr)^{\rho/2} + \Bigl(\sqrt{\frac{d}{n}}\Bigr)^s\Bigr],
\end{align*}
where $C$ is a constant depending on $M,L,m,s.$
If $d\lesssim n^{\alpha}$ for some $\alpha\in (0,1)$ and $s>\frac{1}{1-\alpha},$ it easily follows 
that, for all $B>0,$ 
\begin{align}
\label{upper_n_to_infty}
\limsup_{n\to\infty}\sup_{(f,\theta): \frac{\|f\|_{C^s}}{\sigma_f(\theta)}\leq B}
\frac{\sqrt{n}\|f_k(\hat \theta)-f(\theta)\|_{L_2({\mathbb P}_{\theta})}}
{\sigma_f(\theta)} \leq 1.
\end{align}


The following minimax lower bound will be proved in Section \ref{Sec:lower_bounds}.

\begin{proposition}
\label{min_max_lower_local} 
Suppose Assumption \ref{Assump_main} holds and let $f\in C^s$ for some $s=1+\rho, \rho\in (0,1].$
Then, for all $c>0$ and all $\theta_0\in {\mathbb R}^d,$
\begin{align*}
&
\inf_{\hat T_n}\sup_{\|\theta-\theta_0\|\leq \frac{c}{\sqrt{n}}}\frac{\sqrt{n}\|\hat T_n-f(\theta)\|_{L_2({\mathbb P}_{\theta})}}{\sigma_f(\theta)}
\geq 1- \frac{3\pi}{\sqrt{8m}c}
 -\frac{2}{\sqrt{m}}\frac{\|f\|_{C^s}}{\sigma_f(\theta_0)} 
 \Bigl(\frac{c}{\sqrt{n}}\Bigr)^{\rho},
 \end{align*}
 where the infimum is taken over all estimators $\hat T_n=\hat T_n(X_1,\dots, X_n).$
\end{proposition}


The bound of Proposition \ref{min_max_lower_local} easily implies that, for all $B>0,$
\begin{align*}
\lim_{c\to \infty}\liminf_{n\to\infty}\inf_{(f,\theta_0): \frac{\|f\|_{C^s}}{\sigma_f(\theta_0)}\leq B}\inf_{\hat T_n}\sup_{\|\theta-\theta_0\|\leq \frac{c}{\sqrt{n}}} 
\frac{\sqrt{n}\|\hat T_n- f(\theta)\|_{L_2({\mathbb P}_{\theta})}}{\sigma_f(\theta)}\geq 1.
\end{align*}
Along with \eqref{upper_n_to_infty}, it shows local asymptotic minimaxity of estimator $f_k(\hat \theta).$

The next corollaries will be based on the results by Rio \cite{Rio} on convergence rates in CLT in Wasserstein type distances. For r.v. $\eta_1, \eta_2$
and a convex non-decreasing function $\psi : {\mathbb R}_+\mapsto {\mathbb R}_+$ with $\psi(0)=0,$ define the Wasserstein $\psi$-distance 
between $\eta_1$ and $\eta_2$ as 
\begin{align*}
W_{\psi} (\eta_1,\eta_2) := \inf\Bigl\{\|\eta_1'-\eta_2'\|_{\psi}: \eta_1'\overset{d}{=} \eta_1, \eta_2'\overset{d}{=} \eta_2\Bigr\}.
\end{align*}
For $\psi(u)=u^p, u\geq 0, p\geq 1,$ we will use the notation $W_p$ instead of $W_{\psi}.$ For $\psi=\psi_{\alpha}, \alpha >0,$
we will modify the above definition using a version of $\psi$-norm defined in terms of the moments. Let $\eta_1,\dots, \eta_n$
be i.i.d. copies of a mean zero r.v. $\eta$ with ${\mathbb E}\eta^2 =1.$ It was proved in \cite{Rio} (see Theorem 4.1 and equation (4.3))
that for all $r\in (1,2],$
\begin{align*}
W_r \Bigl(\frac{\eta_1+\dots+\eta_n}{\sqrt{n}}, Z\Bigr)
\lesssim \frac{{\mathbb E}^{1/r}\eta^{r+2}}{\sqrt{n}},
\end{align*}
where $Z\sim N(0,1).$
Applying this bound to $\eta:= \frac{\langle V'(\xi), {\mathcal I}^{-1} f'(\theta)\rangle}{\sigma_f(\theta)}$ yields
\begin{align*}
W_2\Bigl(\frac{1}{\sqrt{n}}\sum_{j=1}^n \langle V'(\xi_j), {\mathcal I}^{-1} f'(\theta)\rangle, \sigma_f(\theta) Z\Bigr)
\lesssim \frac{{\mathbb E}^{1/2} \langle V'(\xi), {\mathcal I}^{-1} f'(\theta)\rangle^4}{\sigma_f(\theta)}n^{-1/2}.
\end{align*}
Thus, Theorem \ref{main_theorem} implies the following corollary.

\begin{corollary}
\label{normal_approx_f_k}
Under the conditions of Theorem \ref{main_theorem}, for all $\theta\in {\mathbb R}^d$
\begin{align*}
&
W_{2, {\mathbb P}_{\theta}}\Bigl(\sqrt{n}(f_k(\hat \theta)-f(\theta)), \sigma_f(\theta) Z\Bigr)
\\ 
&
\leq C_1\frac{{\mathbb E}^{1/2} \langle V'(\xi), {\mathcal I}^{-1} f'(\theta)\rangle^4}{\sigma_f(\theta)} n^{-1/2}
+ C_2\|f\|_{C^s}\Bigl[\sqrt{\frac{c(V)}{n}}\Bigl(\frac{d}{n}\Bigr)^{\rho/2} + \Bigl(\sqrt{\frac{d}{n}}\Bigr)^s\Bigr],
\end{align*}
where $C_1>0$ is an absolute constant and $C_2>0$ is a constant that could depend on $M,L,m,s.$
\end{corollary} 

Using \eqref{eq_V_'_subgaussian}, it is easy to check that, under Assumption \ref{Assump_main}, 
\begin{align*}
{\mathbb E}^{1/2} \langle V'(\xi), {\mathcal I}^{-1} f'(\theta)\rangle^4\lesssim \frac{M}{m^2} \|f'(\theta)\|^2 
\end{align*}
and, in addition, $\sigma_f^2(\theta)\geq M^{-1}\|f'(\theta)\|^2.$ This yields
\begin{align*}
\frac{{\mathbb E}^{1/2} \langle V'(\xi), {\mathcal I}^{-1} f'(\theta)\rangle^4}{\sigma_f(\theta)}\lesssim_{M,m} \|f'(\theta)\|
\lesssim_{M,m}\|f\|_{C^s}.
\end{align*}
Therefore, if $d\leq n^{\alpha}$ for some $\alpha\in (0,1)$ and $s>\frac{1}{1-\alpha},$ then 
\begin{align*}
\sup_{\|f\|_{C^s}\leq 1}\sup_{\theta \in {\mathbb R}^d}W_{2, {\mathbb P}_{\theta}}\Bigl(\sqrt{n}(f_k(\hat \theta)-f(\theta)), \sigma_f(\theta) Z\Bigr)
\to 0\ {\rm as}\ n\to\infty,
\end{align*}
implying asymptotic normality of estimator $f_k(\hat \theta)$ of $f(\theta)$ with rate $\sqrt{n}$ and limit variance $\sigma_f^2(\theta).$
It is also easy to show that, under the same conditions on $d$ and $s,$ we have, for all $B>0,$
\begin{align*}
\sup_{(f,\theta): \frac{\|f\|_{C^s}}{\sigma_f(\theta)}\leq B} 
W_{2, {\mathbb P}_{\theta}}\Bigl(\frac{\sqrt{n}(f_k(\hat \theta)-f(\theta))}{\sigma_f(\theta)}, Z\Bigr)
\to 0\ {\rm as}\ n\to\infty,
\end{align*}
which implies 
\begin{align*}
\sup_{(f,\theta): \frac{\|f\|_{C^s}}{\sigma_f(\theta)}\leq B} 
\sup_{x\in {\mathbb R}}\Bigl|{\mathbb P}_{\theta}\Bigl\{\frac{\sqrt{n}(f_k(\hat \theta)-f(\theta))}{\sigma_f(\theta)}\leq x\Bigr\}-{\mathbb P}\{Z\leq x\}\Bigr| \to 0\ {\rm as}\ n\to\infty.
\end{align*}

It was also proved in \cite{Rio}, Theorem 2.1 that, for i.i.d. copies $\eta_1,\dots, \eta_n$ of mean zero r.v. $\eta$
with ${\mathbb E}\eta^2=1$ and $\|\eta\|_{\psi_1}<\infty$ and for 
some constant $C(\|\eta\|_{\psi_1})<\infty,$ 
\begin{align*}
W_{\psi_1} \Bigl(\frac{\eta_1+\dots+\eta_n}{\sqrt{n}}, Z\Bigr)
\lesssim \frac{C(\|\eta\|_{\psi_1})}{\sqrt{n}}.
\end{align*}
We will again apply this to $\eta:= \frac{\langle V'(\xi), {\mathcal I}^{-1} f'(\theta)\rangle}{\sigma_f(\theta)}.$
In this case, by Lemma \ref{bd_on_V'}, we have 
\begin{align*}
\|\langle V'(\xi), {\mathcal I}^{-1} f'(\theta)\rangle\|_{\psi_1} \lesssim \frac{\sqrt{M}}{m} \|f'(\theta)\|.
\end{align*}
Also, $\sigma_f(\theta)\geq \frac{\|f'(\theta)\|}{\sqrt{M}},$ implying $\|\eta\|_{\psi_1} \lesssim \frac{M}{m}.$
As a result, we get 
\begin{align*}
W_{\psi_1}\Bigl(\frac{1}{n}\sum_{j=1}^n \langle V'(\xi_j), {\mathcal I}^{-1} f'(\theta)\rangle, \frac{\sigma_f(\theta) Z}{\sqrt{n}}\Bigr)
\lesssim_{M,m} \frac{\sigma_f(\theta)}{n}.
\end{align*}
Combining this with the bound of Theorem \ref{main_theorem} yields the following extension of Corollary \ref{cor_L_2}.

\begin{corollary}
\label{cor_bd_psi}
Under the conditions of Theorem \ref{main_theorem}, for all convex non-decreasing functions 
$\psi : {\mathbb R}_+\rightarrow {\mathbb R}_+$ with $\psi(0)=0,$ 
satisfying the condition $\psi(u)\leq \psi_{2/3}(c u), u\geq 0$
for some constant $c>0,$ 
\begin{align*}
&
\sup_{\theta\in {\mathbb R}^d}\Bigl|\|f_k(\hat \theta)-f(\theta)\|_{L_{\psi}({\mathbb P}_{\theta})}- \frac{\sigma_f(\theta)}{\sqrt{n}}\|Z\|_{\psi}\Bigr|
\\
&
\lesssim_{L,M,m,s}
\|f\|_{C^s} \Bigl[\sqrt{\frac{c(V)}{n}}\Bigl(\frac{d}{n}\Bigr)^{\rho/2} + \Bigl(\sqrt{\frac{d}{n}}\Bigr)^s\Bigr].
\end{align*}
\end{corollary}

\begin{remark} 
\normalfont
Similar results were obtained in \cite{Koltchinskii_Zhilova} in the case of Gaussian 
shift models, in \cite{Koltchinskii_Zhilova_2020} in the case of more general Poincar\'e random shift models 
and in \cite{Koltchinskii_2017, Koltchinskii_2018, Koltchinskii_Zhilova_19} in the case of Gaussian models with 
unknown covariance and unknown mean and covariance (the analysis becomes much more involved in the case when the functional depends on unknown covariance). In \cite{Koltchinskii_2020}, the proposed higher order bias reduction method was studied in the case of general models with a high-dimensional parameter $\theta$ for which there exists an estimator $\hat \theta$ admitting high-dimensional normal approximation. 
\end{remark}

\begin{remark}
\normalfont
If ${\mathbb E}\xi=0,$ one can also use $\bar X=\frac{X_1+\dots+X_n}{n}$ as an estimator of $\theta$ and construct the corresponding 
functions $\bar f_k$
based on this estimator. In this case, a bound similar to \eqref{min_max_upper} holds for estimator $\bar f_k(\bar X),$ so, it is also minimax 
optimal. This follows from Theorem 2, \cite{Koltchinskii_Zhilova_2020} along with the bound on Poincar\'e constant $c(V)$ (see Remark \ref{KLS}). Normal approximation of estimator $\bar f_k(\bar X)$ similar to Corollary \ref{normal_approx_f_k} also holds (see \cite{Koltchinskii_Zhilova_2020}). However, the limit variance of estimator $\bar f_k(\bar X)$ is not equal to $\sigma_f^2(\theta),$
it is rather equal to $\langle \Sigma_{\xi} f'(\theta), f'(\theta)\rangle.$ Since $X$ is an unbiased estimator of $\theta$ (note that ${\mathbb E}\xi=0$), it follows from the Cram\'er-Rao bound that 
$\Sigma_X=\Sigma_{\xi}\succeq {\mathcal I}^{-1}.$ 
This fact implies that the limit variance of estimator $\bar f_k(\bar X)$ is suboptimal: 
\begin{align*}
\langle \Sigma_{\xi} f'(\theta), f'(\theta)\rangle\geq \langle {\mathcal I}^{-1} f'(\theta), f'(\theta)\rangle
\end{align*}
and this estimator 
is not asymptotically efficient. This was the main motivation for the development of estimators $f_k(\hat \theta)$ based on the MLE
in the current paper. We conjecture that asymptotic efficiency also holds when the MLE is replaced by Pitman's estimator. Since MLE is defined implicitly as a solution of an optimization problem, there is an additional layer of difficulties in the analysis of the problem comparing with the case of $\bar X.$ Similar problems in the case of log-concave location-scale families seem to be much more challenging.
\end{remark}

\begin{remark}
\normalfont
The proof of Theorem \ref{main_theorem} could be easily modified and, in fact, significantly simplified 
to obtain the following result under somewhat different assumptions than Assumption \ref{Assump_main} (they are stronger in the sense that the eigenvalues of the Hessian $V''(x)$ are assumed to be bounded away from zero uniformly in $x$).
 
\begin{theorem}
\label{main_theorem_AAA}
Suppose $V$ is twice continuously differentiable and, for some $M,m>0,$ $\|V''\|_{L_{\infty}}\leq M$ and $V''(x)\succeq m I_d, x\in {\mathbb R}^d.$ 
Let $f\in C^s$ for some $s=k+1+\rho,$ $k\geq 0,$ $\rho\in (0,1].$
Then, for all $d\lesssim n,$ 
\begin{align*}
&
\sup_{\theta \in {\mathbb R}^d}\Bigl\|f_k(\hat \theta) - f(\theta)- n^{-1}\sum_{j=1}^n \langle V'(\xi_j), {\mathcal I}^{-1} f'(\theta)\rangle\Bigr\|_{L_{\psi_{1}}({\mathbb P}_{\theta})}
\\
&
\lesssim_{M,m,s} \|f\|_{C^s} \Bigl[\frac{1}{\sqrt{n}}\Bigl(\frac{d}{n}\Bigr)^{\rho/2} + \Bigl(\sqrt{\frac{d}{n}}\Bigr)^s\Bigr].
\end{align*}
\end{theorem}
 
 
 This result implies that, under the conditions of Theorem \ref{main_theorem_AAA}, the second claim of Proposition \ref{prop_min_max_upper} holds without 
 any assumptions on $d$ and on $\Sigma_{\xi}$ and that the bound of Corollary \ref{cor_bd_psi} holds for all 
 convex non-decreasing functions 
$\psi : {\mathbb R}_+\mapsto {\mathbb R}_+$ with $\psi(0)=0,$ 
satisfying the condition $\psi(u)\leq \psi_{1}(c u), u\geq 0$
for some constant $c>0.$
 \end{remark}

\section{Error bounds for the MLE}

Our main goal in this section is to obtain upper bounds on the error $\|\hat \theta-\theta\|$ of MLE $\hat \theta.$
Namely, the following result will be proved.

\begin{theorem}
\label{th_bound_on_hat_theta}
Suppose Assumption \ref{Assump_main} holds and let $t\geq 1.$ 
If $d\vee t\leq \gamma n$ for 
\begin{align}
\label{definition_gamma}
\gamma := c\Bigl(\frac{m}{M}\wedge \frac{m^2}{L \sqrt{M}}\Bigr)^2
\end{align}
with a small enough constant $c>0$, then, with probability at least $1-e^{-t}$,
\begin{align*}
\|\hat \theta-\theta\| \lesssim \frac{\sqrt{M}}{m} \Bigl(\sqrt{\frac{d}{n}}\vee \sqrt{\frac{t}{n}}\Bigr).
\end{align*}
\end{theorem}

Several simple facts will be used in the proof.

For a differentiable function $g: {\mathbb R}^d\mapsto {\mathbb R},$ define the remainder of its first order Taylor expansion 
\begin{align*}
S_g(x;h) := g(x+h)-g(x)- \langle g'(x),h\rangle, x,h\in {\mathbb R}^d. 
\end{align*}
The next proposition is straightforward.

\begin{proposition}
\label{remainder_bd}
Let $g$ be twice differentiable.
Then, for all $x, y, h, h' \in {\mathbb R}^d,$ the following properties hold:
\begin{enumerate}[label=(\roman*)]
\item $|S_g(x;h)|\leq \frac{1}{2} \|g''\|_{L_{\infty}}\|h\|^2.$ 
\item $|S_g(x;h)- \frac{1}{2} \langle g''(x)h,h\rangle|\leq \frac{1}{6}\|g''\|_{\rm Lip}\|h\|^3.$
\item $|S_g(x;h)-S_g(x;h')|\leq \frac{1}{2}\|g''\|_{L_{\infty}} \|h-h'\|^2 +\|g''\|_{L_{\infty}} \|h\|\|h-h'\|.$
\item $|S_g(x;h)-S_g(y;h)|\leq \frac{1}{4} \|g''\|_{\rm Lip}\|h\|^2\|x-y\|.$
\end{enumerate}
If $g\in C^s$ for $s=1+\rho,$ $\rho\in (0,1],$ then
\begin{enumerate}[label=(\roman*)]
\setcounter{enumi}{4}
\item $|S_g(x;h)-S_g(x;h')|\lesssim \|g\|_{C^s}(\|h\|^{\rho}\vee \|h'\|^{\rho})\|h-h'\|.$
\end{enumerate}
\end{proposition}


Let $\xi_1,\dots, \xi_n$ be i.i.d. copies of $\xi$ (that is, $\xi_j:= X_j-\theta$). 
Define the following convex functions:
\begin{align}
&
g(h):= {\mathbb E} V(\xi+h),\nonumber
\\
&
\label{def_g_functions}
g_n(h):= n^{-1}\sum_{j=1}^n V(\xi_j + h), h\in {\mathbb R}^d.
\end{align}
Note that ${\mathbb E}g_n(h)=g(h)$ and $g''(0)={\mathcal I}.$

We will need simple probabilistic bounds for r.v. 
$
g_n'(0) = n^{-1} \sum_{j=1}^n V'(\xi_j)
$
and
$
g_n''(0)=n^{-1} \sum_{j=1}^n V''(\xi_j).
$
We start with the following lemma.

\begin{lemma}
\label{bd_on_V'}
For all $u\in {\mathbb R}^d,$ $\langle V'(\xi),u\rangle$ is a subgaussian r.v. with 
\begin{align*}
\|\langle V'(\xi),u\rangle\|_{\psi_2}\lesssim \sqrt{M}\|u\|.
\end{align*}
\end{lemma}

\begin{proof}
For all $k\geq 1,$ we have
\begin{align*}
{\mathbb E} \langle V'(\xi),u\rangle^{2k}= 
\int_{{\mathbb R}^d} \langle V'(x),u\rangle^{2k} e^{-V(x)}dx.
\end{align*}
By Lemma 2.2.1 in \cite{BGVV14}, there are constants $A,B>0$ such that $e^{-V(x)}\leq Ae^{-B\|x\|}$ for all $x\in\mathbb{R}^d$. Moreover, by Assumption \ref{Assump_main}, $V'$ is $M$-Lipschitz which implies that $\|V'(x)\|\leq \|V'(0)\|+ M\|x\|$ for all $x\in\mathbb{R}^d$. Combining these two facts, the above integral is finite for all $k\geq 1$ and we obtain that
\begin{align*}
\int_{{\mathbb R}^d} \langle V'(x),u\rangle^{2k} e^{-V(x)}dx
&=\int_{{\mathbb R}^d} \langle V'(x),u\rangle^{2k-1} \langle V'(x),u\rangle e^{-V(x)}dx 
\\
&
=\int_{{\mathbb R}^d} (2k-1)\langle V'(x),u\rangle^{2k-2} \langle V''(x)u,u\rangle e^{-V(x)}dx, 
\end{align*}
where we used integration by parts in the last equality. 
Therefore,
\begin{align*}
{\mathbb E} \langle V'(\xi),u\rangle^{2k} &\leq (2k-1) M \|u\|^2 \int_{{\mathbb R}^d}  \langle V'(x),u\rangle^{2k-2} e^{-V(x)}dx
\\
&
= (2k-1) M \|u\|^2 {\mathbb E} \langle V'(\xi),u\rangle^{2(k-1)}.
\end{align*}
It follows by induction that 
\begin{align*}
{\mathbb E} \langle V'(\xi),u\rangle^{2k} \leq (2k-1)!! M^k \|u\|^{2k}.
\end{align*}
It is easy to conclude that, for all $p\geq 1,$
\begin{align*}
\|\langle V'(\xi),u\rangle \|_{L_p}\lesssim \sqrt{p} \sqrt{M} \|u\|, 
\end{align*}
implying the claim.

\qed
\end{proof}

An immediate consequence is the following corollary.

\begin{corollary}
\label{cor_g_n'}
For all $t\geq 1,$ with probability at least $1-e^{-t}$
\begin{align*}
\|g_n'(0)\| \lesssim \sqrt{M}\Bigl(\sqrt{\frac{d}{n}}\vee \sqrt{\frac{t}{n}}\Bigr).
\end{align*}
\end{corollary}

\begin{proof}
Let $S^{d-1}$ be the unit sphere in $\mathbb{R}^d$ and let $A\subset S^{d-1}$ be a $1/2$-net with ${\rm card}(A)\leq 5^d.$ Then 
\begin{align*}
\|g_n'(0)\|= \sup_{u\in S^{d-1}} \langle g_n'(0), u\rangle \leq 2\max_{u\in A} \langle g_n'(0), u\rangle.
\end{align*}
By Lemma \ref{bd_on_V'}, we have for all $u\in A,$ 
\begin{align*}
\|\langle g_n'(0), u\rangle\|_{\psi_2} = \Bigl\|n^{-1} \sum_{j=1}^n \langle V'(\xi_j), u\rangle\Bigr\|_{\psi_2}  
\lesssim \frac{\sqrt{M}}{\sqrt{n}},
\end{align*}
implying that with probability at least $1-e^{-t}$
\begin{align*}
|\langle g_n'(0), u\rangle| \lesssim \sqrt{M}\sqrt{\frac{t}{n}}.
\end{align*}
It remains to use the union bound and to replace $t$ by $t+ d \log (5)$.

\qed
\end{proof}

\begin{proposition}
\label{bd_g_n''}
For all $t\geq 1,$ with probability at least $1-e^{-t}$
\begin{align*}
\|g_n''(0)-{\mathcal I}\| \lesssim M\Bigl(\sqrt{\frac{d}{n}}\vee \sqrt{\frac{t}{n}}\Bigr).
\end{align*}
Moreover, 
\begin{align*}
\Bigl\|\|g_n''(0)-{\mathcal I}\|\Bigr\|_{\psi_2} \lesssim M \sqrt{\frac{d}{n}}.
\end{align*}
\end{proposition}

\begin{proof}
Similarly to the proof of Corollary \ref{cor_g_n'}, one can use the fact that $\|\langle V''(\xi)u, v\rangle\|_{\psi_2}\lesssim \|\langle V''(\xi)u, v\rangle\|_{L_{\infty}}\lesssim M, u, v\in S^{d-1}$ and discretization of the unit sphere to prove the first bound.

Moreover, using that $t,d\geq 1$, the first bound implies that
\begin{align*} 
{\mathbb P}\biggl\{\|g_n''(0)-{\mathcal I}\| \geq  C_1 M\sqrt{\frac{d}{n}}\sqrt{t}\biggr\}
\leq e^{-t}, t\geq 1,
\end{align*}
which is equivalent to the second bound.

\qed
\end{proof}

We now turn to the proof of Theorem \ref{th_bound_on_hat_theta}.

\begin{proof}
Note that the minimum of convex function $g$ from \eqref{def_g_functions} is attained at $0$ and also
\begin{align}\label{def_hat_h}
\hat h:= \operatorname{argmin}_{h\in {\mathbb R}^d}\limits g_n(h)= \theta-\hat \theta,
\end{align}
so, to prove Theorem \ref{th_bound_on_hat_theta}, it will be enough to bound $\|\hat h\|.$
We will use the following elementary lemma.

\begin{lemma}
\label{elementary_convex}
Let $q: {\mathbb R}^d\mapsto {\mathbb R}$ be a convex function attaining its minimum at $\bar x\in {\mathbb R}^d.$ For all $x_0\in {\mathbb R}^d$
and $\delta>0,$ the condition $\|\bar x-x_0\|\geq \delta$ implies that 
$
\inf_{\|x-x_0\|=\delta} q(x)-q(x_0) \leq 0.
$
\end{lemma}

\begin{proof}
Indeed, assume that $\|\bar x-x_0\|\geq \delta.$ Clearly, $q(\bar x)-q(x_0)\leq 0.$
Let 
$x^{\ast}= \lambda \bar x+ (1-\lambda) x_0$ with $\lambda :=\frac{\delta}{\|\bar x-x_0\|}.$
Then, $\|x^{\ast}-x_0\|= \delta$ and, 
by convexity of $q,$ $q(x^{\ast})\leq \lambda q(\bar x) +(1-\lambda)q(x_0),$
implying that $q(x^{\ast})-q(x_0)\leq\lambda (q(\bar x)-q(x_0))\leq 0.$

\qed
\end{proof}

If $\|\hat h\|\geq \delta,$ then, by Lemma \ref{elementary_convex}, 
\begin{align}\label{eq_application_elementay_convex}
\inf_{\|h\|=\delta} g_n(h)-g_n(0) \leq 0.
\end{align}
Note that 
\begin{align}
g_n(h)-g_n(0) &= \langle g_n'(0),h \rangle + S_{g_n}(0;h)\nonumber
\\
&
\label{eq_bound_hat_theta_eins}
=\langle g_n'(0),h\rangle + S_{g_n}(0;h)-\frac{1}{2}\langle g_n''(0)h,h\rangle+\frac{1}{2}\langle g_n''(0)h,h\rangle.
\end{align}
For $\|h\|=\delta,$ we have, by Assumption \ref{Assump_main}, (ii),
\begin{align*}
    \langle g_n''(0)h,h\rangle=\langle g_n''(0)h,h\rangle-\langle {\mathcal I} h,h\rangle+\langle {\mathcal I} h,h\rangle\geq m\delta^2-\delta^2\|g_n''(0)-\mathcal{I}\|
\end{align*}
and, by Proposition \ref{remainder_bd}, (ii),
\begin{align*}
S_{g_n}(0;h)- \frac{1}{2} \langle g_n''(0)h,h\rangle \geq -\frac{L}{2} \delta^3.
\end{align*}
Inserting these inequalities into \eqref{eq_bound_hat_theta_eins} and using \eqref{eq_application_elementay_convex} we can conclude that if $\|\hat h\|\geq \delta$, then
\begin{align}
\label{main_bound}
  \|g_n'(0)\| \delta + \frac{\delta^2}{2}\|g_n''(0)-\mathcal{I}\|\geq \frac{m}{2}\delta^2 - \frac{L}{2} \delta^3.
\end{align}
To complete the proof, assume that the bound of Corollary \ref{cor_g_n'} holds with constant $C_1\geq 1$ and the bound of Proposition \ref{bd_g_n''} holds with constant $C_2\geq 1.$  
If constant $c$ in the definition of $\gamma$ is small enough, then 
the condition $d\vee t\leq \gamma n$ implies that 
\begin{align*}
C\Bigl(\sqrt{\frac{d}{n}}\vee \sqrt{\frac{t}{n}}\Bigr)\leq \frac{m}{M}\wedge \frac{m^2}{L \sqrt{M}}
\end{align*}
with $C:=(16C_1) \vee (4C_2).$ Moreover, let 
\begin{align*}
\delta:= 4C_1\frac{\sqrt{M}}{m} \Bigl(\sqrt{\frac{d}{n}}\vee \sqrt{\frac{t}{n}}\Bigr)
\end{align*}
Then, $\delta\leq \frac{m}{4L}$ and, on the event 
\begin{align*}
E:=\Bigl\{\|g_n''(0)-\mathcal{I}\| \leq C_2 M \Bigl(\sqrt{\frac{d}{n}}\vee \sqrt{\frac{t}{n}}\Bigr)\Bigr\},
\end{align*}
bound \eqref{main_bound} implies that
$
\delta \leq \frac{4}{m} \|g_n'(0)\|.
$
Note also that, by Proposition~\ref{bd_g_n''}, ${\mathbb P}(E^c)\leq e^{-t}$. By Corollary \ref{cor_g_n'}, the event $\{\delta \leq \frac{4}{m} \|g_n'(0)\|\}$ occurs with probability at most $e^{-t}.$
Recall that bound \eqref{main_bound} follows from $\|\hat \theta-\theta\|=\|\hat h\|\geq \delta.$
Thus, with probability at least $1-2e^{-t},$ $\|\hat \theta-\theta\|\leq \delta.$ 
It remains to adjust the constants in order to replace the probability bound $1-2e^{-t}$ with $1-e^{-t}.$

\qed
\end{proof}

The following fact will be also useful.

\begin{corollary}
\label{bd_on_hat_theta_wedge_cor}
Suppose Assumption \ref{Assump_main} holds and that $d\leq \gamma n,$
where $\gamma = c\big(\frac{m}{M}\wedge \frac{m^2}{L\sqrt{M}}\big)^2$ with a small enough constant $c>0.$
Then
\begin{align*}
\Bigl\|\|\hat \theta-\theta\| \wedge \frac{m}{12L}\Bigr\|_{\psi_2}
\lesssim  \frac{\sqrt{M}}{m} \sqrt{\frac{d}{n}}+\frac{m}{L\sqrt{\gamma}}\frac{1}{\sqrt{n}}.
\end{align*}
\end{corollary}

\begin{proof}
First, for $d\leq \gamma n$, Theorem \ref{th_bound_on_hat_theta} can be formulated as
\begin{align*}
{\mathbb P}\biggl\{\|\hat \theta-\theta\| >  
C_1 \Bigl(\frac{\sqrt{M}}{m} \sqrt{\frac{d}{n}}\vee\frac{\sqrt{M}}{m} \sqrt{\frac{t}{n}}\Bigr)\biggr\}\leq e^{-t}, t\in [1,\gamma n].
\end{align*}
This implies that
\begin{align*}
{\mathbb P}\biggl\{\|\hat \theta-\theta\|  >  
\Bigr(C_1 \frac{\sqrt{M}}{m}\sqrt{\frac{d}{n}}\vee \frac{m}{12L\sqrt{\gamma}}\frac{1}{\sqrt{n}}\Bigl) \sqrt{t}\biggr\}\leq e^{-t}, t\in [1,\gamma n],
\end{align*}
using that $t\geq 1$ and $\sqrt{\gamma}\leq \frac{1}{12C_1}\frac{m^2}{L\sqrt{M}}$ for $c$ sufficiently small. It follows that 
\begin{align*}
{\mathbb P}\biggl\{\|\hat \theta-\theta\|\wedge \frac{m}{12L} >  
\Bigr(C_1 \frac{\sqrt{M}}{m}\sqrt{\frac{d}{n}}\vee \frac{m}{12L\sqrt{\gamma}}\frac{1}{\sqrt{n}}\Bigl) \sqrt{t}\biggr\}\leq e^{-t}, t\geq 1, 
\end{align*}
which is equivalent to the claim.

\qed
\end{proof}

\section{Concentration bounds}
\label{Sec:conc}

In this section, we prove concentration inequalities for $f(\hat \theta),$ where $f$ is a smooth function on ${\mathbb R}^d.$
Namely, we will prove the following result. 

\begin{theorem}
\label{smooth_func_conc}
Let $f\in C^s$ for some $s=1+\rho,$ $\rho\in (0,1].$
Suppose that $d\leq \gamma n,$ where $\gamma := c\bigl(\frac{m}{M}\wedge \frac{m^2}{L \sqrt{M}}\bigr)^2$ with a small enough $c>0.$
Then 
\begin{align*}
&
\sup_{\theta \in {\mathbb R}^d}\Bigl\|f(\hat \theta) - {\mathbb E}_{\theta} f(\hat \theta)- n^{-1}\sum_{j=1}^n \langle V'(\xi_j), {\mathcal I}^{-1} f'(\theta)\rangle\Bigr\|_{L_{\psi_{2/3}}({\mathbb P}_{\theta})}
\\
&
\lesssim_{M,L,m} \sqrt{c(V)} \|f\|_{C^s} \frac{1}{\sqrt{n}}\Bigl(\frac{d}{n}\Bigr)^{\rho/2}.
\end{align*}
\end{theorem}


To derive concentration bounds for $f(\hat \theta),$ we need to bound local Lipschitz constants of estimator 
$\hat \theta (X_1,\dots, X_n)$ as a function of its variables. A good place to start is to show the continuity 
of this function. The following fact is, probably, well known. We give its proof for completeness.

\begin{proposition}
\label{cont_hat_theta} 
Suppose that $V$ is strictly convex. 
Then, MLE $\hat \theta (x_1,\dots, x_n)$ exists and is unique 
for all $(x_1,\dots, x_n)\in {\mathbb R}^d\times \dots \times {\mathbb R}^d$ and the function 
$$
{\mathbb R}^d\times \dots \times {\mathbb R}^d \ni (x_1,\dots, x_n)\mapsto \hat \theta (x_1,\dots, x_n)\in {\mathbb R}^d
$$
is continuous.
\end{proposition}

\begin{proof}
Let $(x_1,\dots, x_n)\in {\mathbb R}^d\times \dots \times {\mathbb R}^d, (x_1^{(k)},\dots, x_n^{(k)})\in {\mathbb R}^d\times \dots \times {\mathbb R}^d, k\geq 1$ 
and $(x_1^{(k)},\dots, x_n^{(k)})\to (x_1,\dots, x_n)$ as $k\to \infty.$ Define
\begin{align*}
p(\theta):= n^{-1}\sum_{j=1}^n V(x_j-\theta),\ p_k(\theta):= n^{-1}\sum_{j=1}^n V(x_j^{(k)}-\theta), \theta\in {\mathbb R}^d, k\geq 1.
\end{align*}
By continuity of $V,$ $p_k(\theta)\to p(\theta)$ as $k\to \infty$ for all $\theta\in {\mathbb R}^d.$ Since $p_k$ and $p$ are convex,
this implies the uniform convergence on all compact subsets of ${\mathbb R}^d.$

If $\|\hat \theta (x_1^{(k)}, \dots, x_n^{(k)})-\hat \theta(x_1,\dots ,x_n)\| \geq \delta,$ then, by Lemma \ref{elementary_convex}, 
\begin{align*}
\inf_{\|\theta-\hat \theta(x_1,\dots, x_n)\|=\delta} p_k(\theta)-p_k(\hat \theta(x_1,\dots, x_n)) \leq 0
\end{align*}
By the uniform convergence of $p_k$ to $p$ on compact sets,
\begin{align*}
\inf_{\|\theta-\hat \theta(x_1,\dots, x_n)\|=\delta} p_k(\theta)-p_k(\hat \theta(x_1,\dots, x_n)) \to \inf_{\|\theta-\hat \theta(x_1,\dots, x_n)\|=\delta} p(\theta)-p(\hat \theta(x_1,\dots, x_n))
\end{align*}
as $k\to \infty.$ Due to strict convexity, the minimum $\hat \theta(x_1,\dots, x_n)$ of $p(\theta)$ exists and is unique (see the argument in the introduction), and 
\begin{align*}
\inf_{\|\theta-\hat \theta(x_1,\dots, x_n)\|=\delta} p(\theta)-p(\hat \theta(x_1,\dots, x_n))>0,
\end{align*}
implying that $\|\hat \theta (x_1^{(k)}, \dots, x_n^{(k)})-\hat \theta(x_1,\dots, x_n)\| <\delta$ for all large enough $k$ and thus $\hat \theta (x_1^{(k)}, \dots, x_n^{(k)})\to \hat \theta(x_1,\dots, x_n)$ 
as $k\to \infty.$

\qed
\end{proof}

Note that the continuity of $\hat \theta$ also follows from the implicit function theorem in the case when $V$ is twice differentiable with $V''$ being positively definite throughout $\mathbb{R}^d.$

We will now study Lipschitz continuity properties of $\hat \theta$ as a function of the data $X_1,\dots, X_n$ needed to prove concentration inequalities.

\begin{proposition}
\label{Lip_basic}
Let 
\begin{align*}
A_1:=
\Bigl\{(x_1,\dots, x_n)\in {\mathbb R}^d\times \dots \times {\mathbb R}^d: 
\|\hat \theta (x_1,\dots, x_n)-\theta\|\leq \frac{m}{12 L}
\Bigr\}
\end{align*}
and 
\begin{align*}
A_2:= \Bigl\{(x_1,\dots, x_n)\in {\mathbb R}^d\times \dots \times {\mathbb R}^d: 
\Bigl\|n^{-1}\sum_{j=1}^n V''(x_j-\theta) - \mathcal{I}\Bigr\|\leq \frac{m}{4}
\Bigr\}
\end{align*}
and let $A:=A_1\cap A_2.$
Then the function $A\ni(x_1,\dots, x_n)\mapsto \hat \theta(x_1,\dots, x_n)$ is Lipschitz with constant $\frac{4M}{m \sqrt{n}}:$
for all $(x_1,\dots, x_n), (\tilde x_1,\dots, \tilde x_n)\in A,$
\begin{align*}
\|\hat \theta(x_1,\dots, x_n)-\hat \theta(\tilde x_1,\dots, \tilde x_n)\|\leq \frac{4M}{m \sqrt{n}} \Bigl(\sum_{j=1}^n \|x_j-\tilde x_j\|^2\Bigr)^{1/2}.
\end{align*}
\end{proposition}

\begin{proof}
Due to equivariance, we have
\begin{align*}
    \hat \theta(x_1,\dots, x_n)-\hat \theta(\tilde x_1,\dots, \tilde x_n)=\hat \theta(\xi_1,\dots, \xi_n)-\hat \theta(\tilde \xi_1,\dots, \tilde \xi_n)
\end{align*}
with $\xi_j=x_j-\theta$ and  $\tilde \xi_j=\tilde x_j-\theta.$ Hence, if we abbreviate $\hat h=\theta-\hat \theta(\xi_1,\dots, \xi_n)$ and $\tilde h=\theta-\hat \theta(\tilde \xi_1,\dots, \tilde \xi_n)$, then we have $g_n'(\hat h)=0$ with $g_n$ from \eqref{def_g_functions} and $\tilde g_n'(\tilde h)=0$ with
\begin{align*}
\tilde g_n(h):= n^{-1}\sum_{j=1}^n V(\tilde \xi_j+h), h\in {\mathbb R}^d.
\end{align*}
Recall that $g'(0)=0$ and $g''(0)={\mathcal I}$. By the first order Taylor expansion for function $g',$
$
g'(h) =  {\mathcal I} h + r(h), 
$
where 
\begin{align*}
r(h):=\int_{0}^{1} (g''(\lambda h)-g''(0))d\lambda\ h
\end{align*}
is the remainder. Therefore,
\begin{align*}
g'(\hat h) ={\mathcal I} \hat h + r(\hat h),
\end{align*}
implying that 
\begin{align}
\label{represent_hat_h}
\nonumber
\hat h &= {\mathcal I}^{-1} (g'(\hat h)-g_n'(\hat h)) - {\mathcal I}^{-1} r(\hat h) 
\\
&
\nonumber
={\mathcal I}^{-1} (g'(0)-g_n'(0)) + {\mathcal I}^{-1} q_n(\hat h)
-{\mathcal I}^{-1} r(\hat h)
\\
&
=-{\mathcal I}^{-1} g_n'(0) + {\mathcal I}^{-1} q_n(\hat h)
-{\mathcal I}^{-1} r(\hat h),
\end{align}
where 
\begin{align*}
q_n(h):= (g_n-g)'(h)-(g_n-g)'(0)=\int_{0}^1 (g_n''-g'')(\lambda h) d\lambda\ h.
\end{align*}
Similarly, we have $\tilde h=-{\mathcal I}^{-1} \tilde g_n'(0) + {\mathcal I}^{-1} \tilde q_n(\tilde h)
-{\mathcal I}^{-1} r(\tilde h)$ with $\tilde q_n(h):= (\tilde g_n-g)'(h)-(\tilde g_n-g)'(0).$ Using these representations we now bound the difference between $\hat h$ and $\tilde h.$


First note that 
\begin{align}
\label{eins}
\nonumber
&
\|g_n'(0)-\tilde g_n'(0)\| \leq n^{-1}\sum_{j=1}^n\|V'(\xi_j)-V'(\tilde \xi_j)\|
\\
&
 \leq M n^{-1}\sum_{j=1}^n \|\xi_j-\tilde \xi_j\|
\leq \frac{M}{\sqrt{n}}\Bigl(\sum_{j=1}^n \|x_j-\tilde x_j\|^2\Bigr)^{1/2}. 
\end{align}

Also,
\begin{align*}
q_n(\hat h)-\tilde q_n(\tilde h)&= ((g_n''-g'')(0))(\hat h-\tilde h)
\\
&
+
\int_{0}^1 ((g_n''-g'')(\lambda \hat h)-(g_n''-g'')(0)) d\lambda\ (\hat h-\tilde h) 
\\
&
+ \int_{0}^1 [(g_n''-g'')(\lambda \hat h)- (g_n''-g'')(\lambda \tilde h)]d\lambda \ \tilde h
\\
&
+\int_{0}^1 (g_n''-\tilde g_n'')(\lambda \tilde h) d\lambda\ \tilde h.
\end{align*}
Since, by Assumption \ref{Assump_main}, $V''$ is Lipschitz with constant $L,$ the function $h\mapsto g''(h)= {\mathbb E} V''(\xi +h)$ satisfies 
the Lipschitz condition with the same constant $L$ and the function $h\mapsto (g_n''-g'')(h)$ is Lipschitz with constant at most $2L.$ 
In addition, 
\begin{align*}
\|(g_n''-\tilde g_n'')(\lambda \tilde h)\| &\leq n^{-1}\sum_{j=1}^n \|V''(\xi_j + \lambda \tilde h)-V''(\tilde \xi_j + \lambda \tilde h)\|
\\
&
\leq \frac{L}{n} \sum_{j=1}^n \|\xi_j-\tilde \xi_j\|\leq \frac{L}{\sqrt{n}} \biggl(\sum_{j=1}^n \|x_j-\tilde x_j\|^2\biggr)^{1/2}.
\end{align*}

Therefore, we easily get
\begin{align}
\label{zwei}
\nonumber
\|q_n(\hat h)-\tilde q_n(\tilde h)\| &\leq 
\|g_n''(0)-g''(0)\|\|\hat h-\tilde h\| + L(\|\hat h\|+\|\tilde h\|)\|\hat h-\tilde h\|
\\
&
+ \frac{L}{\sqrt{n}} \|\tilde h\| \biggl(\sum_{j=1}^n \|x_j-\tilde x_j\|^2\biggr)^{1/2}.
\end{align}

Similarly, note that 
\begin{align*}
r(\hat h)-r(\tilde h)= \int_{0}^{1} (g''(\lambda \hat h)-g''(\lambda \tilde h))d\lambda\ \hat h + \int_{0}^{1} (g''(\lambda \tilde h)-g''(0))d\lambda\ (\hat h-\tilde h)
\end{align*}
which implies the following bound:
\begin{align}
\label{drei}
\|r(\hat h)-r(\tilde h)\| 
\leq \frac{L}{2} (\|\hat h\|+\|\tilde h\|) \|\hat h-\tilde h\|.
\end{align}

It follows from \eqref{eins}, \eqref{zwei} and \eqref{drei} that 
\begin{align*}
\|\hat h-\tilde h\| &\leq 
\frac{1}{m}\biggl(\Bigl(\frac{M}{\sqrt{n}}+\frac{L}{\sqrt{n}} \|\tilde h\| \Bigr)\Bigl(\sum_{j=1}^n \|x_j-\tilde x_j\|^2\Bigr)^{1/2} 
\\
&
+ 
\|g_n''(0)-g''(0)\|\|\hat h-\tilde h\| + \frac{3}{2}L(\|\hat h\|+\|\tilde h\|)\|\hat h-\tilde h\|
\biggr).
\end{align*}
If $\|g_n''(0)-g''(0)\| \leq \frac{m}{4}$ and $\|\hat h\|\vee \|\tilde h\| \leq \frac{m}{12 L},$ we easily conclude that 
\begin{align*}
\|\hat h-\tilde h\| &\leq \frac{4M}{m \sqrt{n}} \Bigl(\sum_{j=1}^n \|x_j-\tilde x_j\|^2\Bigr)^{1/2}, 
\end{align*}
which completes the proof.

\qed
\end{proof}

Since the Lipschitz condition holds for $\hat \theta$ only on set $A,$ it will be convenient for our purposes to replace $\hat \theta$ with its ``smoothed truncated" 
version $\check \theta$ that is Lipschitz in the whole space. To this end, let $\phi : {\mathbb R}\mapsto [0,1]$ be defined as follows: $\phi(s)=1, s\leq 1,$ $\phi(s)=0, s\geq 2$
and $\phi (s) = 2-s, s\in (1,2).$ Clearly, $\phi$ is Lipschitz with constant $1.$ By Theorem \ref{th_bound_on_hat_theta}, $\|\hat \theta -\theta\| \leq \frac{m}{24L}$
with probability at least $1-e^{-\gamma n},$ where $\gamma := c\bigl(\frac{m}{M}\wedge \frac{m^2}{L \sqrt{M}}\bigr)^2$ with a small enough constant $c>0$ and it is assumed that $d\leq \gamma n.$ Similarly, it follows from Proposition \ref{bd_g_n''} that $\|g_n''(0)-g''(0)\|\leq \frac{m}{8}$ with probability at least 
$1-e^{-\beta n},$ where $\beta = c \bigl(\frac{m}{M}\bigr)^2$ for a small enough constant $c>0$ (and under the assumption that $d\leq \beta n.$  
Clearly, we can assume that $\beta\geq \gamma,$ so, both properties hold with probability at least $1-2e^{-\gamma n}$ provided that $d\leq \gamma n.$
Define 
\begin{align*}
&
\varphi (x_1,\dots, x_n) := \phi \Bigl(\frac{24L}{m}\|\hat \theta(x_1,\dots,x_n)-\theta\|\Bigr),
\\
&
\psi(x_1,\dots, x_n) := \phi \Bigl(\frac{8}{m}\|g_n''(0)(x_1-\theta,\dots, x_n-\theta)-g''(0)\|\Bigr).
\end{align*}
and let 
\begin{align*}
\check \theta :=  (1-\varphi \psi) \theta + \varphi \psi \hat \theta.
\end{align*}
Note that $\check \theta -\theta = (\hat \theta-\theta) \varphi \psi$ and $\check \theta= \hat \theta$ on the event $\{\varphi =\psi =1\}$
of probability at least $1-2e^{-\gamma n}.$ 

\begin{proposition}
\label{bd_on_bias}
If $d\leq \gamma n,$ then 
\begin{align*}
\|{\mathbb E}_{\theta} \check \theta -\theta\| \lesssim_{M,L,m} \frac{d}{n}
\qquad\text{and}\qquad
{\mathbb E}_{\theta}^{1/2} \|\check \theta -\theta\|^2 \lesssim_{M,L,m} \sqrt{\frac{d}{n}}.
\end{align*}
\end{proposition}

\begin{proof}
By representation \eqref{represent_hat_h} and the fact that $g'(0)=0$, \begin{align}
\label{repr_check_theta'''}
\nonumber
\check \theta-\theta &= (\hat \theta-\theta) \varphi \psi=
-\hat h \varphi \psi 
\\
&
={\mathcal I}^{-1}g_n'(0) -  {\mathcal I}^{-1}g_n'(0) (1-\varphi \psi) 
- {\mathcal I}^{-1}q_n(\hat h) \varphi \psi
+ {\mathcal I}^{-1}r(\hat h) \varphi \psi.
\end{align}
Using Corollary \ref{cor_g_n'}, we get
\begin{align*}
 \|\|{\mathcal I}^{-1}g_n'(0) (1-\varphi \psi)\|\|_{L_2}
&\leq \|\|{\mathcal I}^{-1}g_n'(0)\|\|_{L_4}
\|I(\varphi \psi \neq 1)\|_{L_4} 
\\
&
\lesssim \frac{1}{m} \|\|g_n'(0)\|\|_{L_4} e^{-\gamma n/4} \lesssim \frac{\sqrt{M}}{m} \sqrt{\frac{d}{n}}e^{-\gamma n/4}.
\end{align*}
Note also that 
\begin{align*}
\|q_n(\hat h)\| =\Bigl\|\int_{0}^1 (g_n''-g'')(\lambda \hat h) d\lambda\ \hat h\Bigr\|\leq \|g_n''(0)-g''(0)\|\|\hat h\|+L\|\hat h\|^2,
\end{align*}
such that we also have 
\begin{align*}
{\mathbb E}^{1/2}\|{\mathcal I}^{-1}q_n(\hat h) \varphi \psi\|^2 &\leq \frac{1}{m} {\mathbb E}^{1/2}\|g_n''(0)-g''(0)\|^2 
\Bigl(\|\hat h\|\wedge \frac{m}{12L}\Bigr)^2
+ \frac{L}{m} {\mathbb E}^{1/2}\Bigl(\|\hat h\|\wedge \frac{m}{12L}\Bigr)^4
\\
&
\leq 
\frac{1}{m} {\mathbb E}^{1/4}\|g_n''(0)-g''(0)\|^4 {\mathbb E}^{1/4}\Bigl(\|\hat h\|\wedge \frac{m}{12L}\Bigr)^4
+ \frac{L}{m} {\mathbb E}^{1/2}\Bigl(\|\hat h\|\wedge \frac{m}{12L}\Bigr)^4.
\end{align*}
Using the second bound of Proposition \ref{bd_g_n''} and the bound of Corollary \ref{bd_on_hat_theta_wedge_cor},
we get 
\begin{align*}
{\mathbb E}^{1/2}\|{\mathcal I}^{-1}q_n(\hat h) \varphi \psi\|^2 \lesssim_{M,L,m} \frac{d}{n}.
\end{align*}
Similarly, we can show that 
\begin{align*}
{\mathbb E}^{1/2}\|{\mathcal I}^{-1}r(\hat h) \varphi \psi\|^2 \lesssim_{M,L,m} \frac{d}{n},
\end{align*}
using the fact that 
\begin{align*}
\|{\mathcal I}^{-1} r(\hat h) \varphi \psi\| \leq \frac{L}{2m} \Bigl(\|\hat h\|\wedge \frac{m}{12L}\Bigr)^2
\end{align*}
and  the bound of Corollary \ref{bd_on_hat_theta_wedge_cor}.

The above bounds and representation \eqref{repr_check_theta'''} imply that 
\begin{align*}
\|{\mathbb E}_{\theta} \check \theta -\theta\| \lesssim_{M,L,m} \frac{d}{n}.
\end{align*}
Using also Corollary \ref{cor_g_n'}, we get 
\begin{align*}
{\mathbb E}_{\theta}^{1/2} \|\check \theta -\theta\|^2 \lesssim_{M,L,m} \sqrt{\frac{d}{n}}.
\end{align*}

\qed
\end{proof}

\begin{proposition}
\label{Lip_check_theta}
The function $(x_1,\dots, x_n)\mapsto \check \theta(x_1,\dots, x_n)$ is Lipschitz with constant $\lesssim \frac{M}{m \sqrt{n}}:$
for all $(x_1,\dots, x_n), (\tilde x_1,\dots, \tilde x_n)\in {\mathbb R}^d \times \dots \times {\mathbb R}^d,$
\begin{align*}
\|\check \theta(x_1,\dots, x_n)-\check \theta(\tilde x_1,\dots, \tilde x_n)\|\lesssim \frac{M}{m \sqrt{n}} \Bigl(\sum_{j=1}^n \|x_j-\tilde x_j\|^2\Bigr)^{1/2}.
\end{align*}
\end{proposition}

\begin{proof}
By Proposition \ref{Lip_basic}, on the  set $A,$ $\hat \theta$ is Lipschitz with constant $\frac{4M}{m \sqrt{n}}.$
This implies that function $\varphi$ is also Lipschitz on the same set with constant $\frac{24L}{m}\frac{4M}{m \sqrt{n}}.$
Note also that 
\begin{align*}
\|g_n''(0)-\tilde g_n''(0)\| \leq \frac{L}{\sqrt{n}} \biggl(\sum_{j=1}^n \|x_j-\tilde x_j\|^2\biggr)^{1/2},
\end{align*}
implying that $\psi$ is a Lipschitz function (on the whole space) with constant $\frac{8}{m}\frac{L}{\sqrt{n}}.$
Using also the fact that $\varphi$ and $\psi$ are both bounded by $1$ and $\|\hat \theta-\theta\|\leq \frac{m}{12L}$ on the 
set $\{\varphi\neq 0\},$ it is easy to conclude that $\check \theta$ is Lipschitz on $A$ with constant 
\begin{align*} 
\lesssim \frac{4M}{m \sqrt{n}}+   \frac{m}{12L}\frac{24L}{m}\frac{4M}{m \sqrt{n}}+ \frac{m}{12L}\frac{8}{m}\frac{L}{\sqrt{n}} \lesssim \frac{M}{m\sqrt{n}}.
\end{align*}

It remains to consider the case when $\varphi (x_1,\dots, x_n)\in A$ and $\varphi (\tilde x_1,\dots, \tilde x_n)\in A^c$ 
(the case when both points are in $A^c$ is 
trivial). In this case, define $x_j^{\lambda} =\lambda x_j+ (1-\lambda)\tilde x_j, \lambda \in [0,1], j=1,\dots, n.$
Note that $A$ is a closed set (by continuity of both $\hat \theta(x_1,\dots, x_n)$ and $n^{-1}\sum_{j=1}^n V''(x_j-\theta)$). 
If $\bar \lambda$ denotes the supremum of those $\lambda$ for which $(x_1^{\lambda}, \dots, x_n^{\lambda})\in A,$
then $(x_1^{\bar \lambda}, \dots, x_n^{\bar \lambda})\in \partial A,$ $(\varphi \psi) (x_1^{\bar \lambda}, \dots, x_n^{\bar \lambda})=0$ 
and $\check \theta (x_1^{\bar \lambda}, \dots, x_n^{\bar \lambda})=0= \check \theta(\tilde x_1,\dots, \tilde x_1).$
Therefore,
\begin{align*}
&
\|\check \theta(x_1,\dots, x_n)- \check \theta(\tilde x_1, \dots, \tilde x_n)\| = \|\check \theta(x_1,\dots, x_n)- \check \theta(x_1^{\bar \lambda}, \dots, x_n^{\bar \lambda})\|
\\
&
\lesssim \frac{M}{m \sqrt{n}} \biggl(\sum_{j=1}^n \|x_j-x_j^{\bar \lambda}\|^2\biggr)^{1/2} \lesssim \frac{M}{m\sqrt{n}} \biggl(\sum_{j=1}^n \|x_j-\tilde x_j\|^2\biggr)^{1/2},
\end{align*}
where we use the fact that point $(x_1^{\bar \lambda}, \dots, x_n^{\bar \lambda})$ is in the line segment between $(x_1,\dots, x_n)$ and $(\tilde x_1,\dots, \tilde x_n).$

The Lipschitz condition for $\check \theta (x_1,\dots, x_n)$ now follows. 

\qed

\end{proof}

We will now consider concentration properties of linear forms $\langle \check \theta-\theta, w\rangle, w\in {\mathbb R}^d.$
The following result will be proved.

\begin{theorem}
\label{lin_func_conc}
Suppose $d\leq \gamma n,$ where $\gamma := c\bigl(\frac{m}{M}\wedge \frac{m^2}{L \sqrt{M}}\bigr)^2$ with a small enough $c>0$.
Then 
\begin{align*}
&
\sup_{\theta\in {\mathbb R}^d}\Bigl\|\langle \check \theta - \theta, w\rangle -{\mathbb E}\langle \check \theta - \theta, w\rangle - 
n^{-1}\sum_{j=1}^n \langle V'(\xi_j),{\mathcal I}^{-1} w\rangle \Bigr\|_{L_{\psi_{2/3}}({\mathbb P}_{\theta})}
\\
&
\lesssim  \sqrt{c(V)} \Bigl(\frac{M^2}{m^2} 
 + \frac{M^{3/2} L}{m^3}\Bigr) \frac{1}{\sqrt{n}}\sqrt{\frac{d}{n}}\|w\|.
\end{align*}
\end{theorem}

\begin{remark}
\normalfont
Some concentration bounds for linear forms of MLE could be found in \cite{Miao}.
\end{remark}

\begin{proof}
Using representation \eqref{represent_hat_h} and the fact that $g'(0)=0$, we get 
\begin{align}
\label{repr_lin_form_check_theta}
\nonumber
&\langle \check \theta-\theta, w\rangle = \langle \hat \theta-\theta, w\rangle \varphi \psi=
-\langle \hat h, w\rangle \varphi \psi 
\\
&
=\langle g_n'(0),u \rangle -  \langle g_n'(0),u \rangle (1-\varphi \psi) 
- \langle q_n(\hat h), u\rangle \varphi \psi
+ \langle r(\hat h), u\rangle \varphi \psi,
\end{align}
where $u={\mathcal I}^{-1} w.$
Since 
\begin{align*}
    \langle g_n'(0),u \rangle=n^{-1}\sum_{j=1}^n \langle V'(\xi_j),{\mathcal I}^{-1} w\rangle
\end{align*}
has zero mean, it is enough to study the concentration of three other terms in the right hand side of \eqref{repr_lin_form_check_theta}.
The first of these terms is $\langle g_n'(0),u \rangle (1-\varphi \psi)$ and we have 
\begin{align*}
&
\|\langle g_n'(0),u \rangle (1-\varphi \psi)\|_{\psi_1} 
\leq \|\langle g_n'(0),u \rangle\|_{\psi_2}  \|1-\varphi \psi\|_{\psi_2}
\\
&
\leq \|\langle g_n'(0),u \rangle\|_{\psi_2}\|I(\varphi\psi\neq 1)\|_{\psi_2} \lesssim \frac{\sqrt{M}}{\sqrt{n}}\frac{1}{\sqrt{\gamma n}}\|u\|
\lesssim \sqrt{\frac{M}{\gamma}}\frac{1}{n} \|u\|,  
\end{align*}
where we used Lemma \ref{bd_on_V'} and the fact that ${\mathbb P}\{\varphi\psi\neq 1\}\leq 2e^{-\gamma n}$ with $\gamma$ from the statement of Theorem \ref{lin_func_conc}. Clearly, we also have 
\begin{align}
\label{odin}
&
\|\langle g_n'(0),u \rangle (1-\varphi \psi)- {\mathbb E}\langle g_n'(0),u \rangle (1-\varphi \psi)\|_{\psi_1} 
\lesssim \sqrt{\frac{M}{\gamma}}\frac{1}{n}\|u\|.
\end{align}

For two other terms in the right hand side of \eqref{repr_lin_form_check_theta}, we will provide bounds on their local Lipschitz
constants. 
It follows from bound \eqref{zwei} and the bound of Proposition \ref{Lip_basic} that, for all $(x_1,\dots, x_n), (\tilde x_1,\dots, \tilde x_n)\in A,$ 
\begin{align*}
&
\|q_n(\hat h)-\tilde q_n(\tilde h)\| 
\\
&
\lesssim 
\biggl(\frac{M }{m\sqrt{n}}\|g_n''(0)-g''(0)\|+ \frac{M L}{m\sqrt{n}}(\|\hat h\|+\|\tilde h\|)\biggr)\biggl(\sum_{j=1}^n \|x_j-\tilde x_j\|^2\biggr)^{1/2}. 
\end{align*}
Recall that function $\varphi$ is Lipschitz on $A$ with constant  $\frac{24L}{m}\frac{4M}{m \sqrt{n}}$ and function $\psi$ is Lipschitz on the whole 
space with constant $\frac{8}{m}\frac{L}{\sqrt{n}}.$ 
Note also that 
\begin{align*}
\|q_n(\hat h)\| =\Bigl\|\int_{0}^1 (g_n''-g'')(\lambda \hat h) d\lambda\ \hat h\Bigr\|\leq \|g_n''(0)-g''(0)\|\|\hat h\|+L\|\hat h\|^2.
\end{align*}
Since, on set $A,$ $\|\hat h\|\leq \frac{m}{12 L},$ we get 
\begin{align*}
\|q_n(\hat h)\| \leq \frac{m}{12L}\|g_n''(0)-g''(0)\|+\frac{m}{12}\|\hat h\|.
\end{align*}
Denoting $\varphi := \varphi (x_1,\dots, x_n), \tilde \varphi := \varphi (\tilde x_1,\dots, \tilde x_n), \psi:=\psi(x_1,\dots, x_n), \tilde \psi:=\psi(\tilde x_1,\dots, \tilde x_n),$ it easily follows from the facts mentioned above that, for all $(x_1,\dots, x_n), (\tilde x_1,\dots, \tilde x_n)\in A,$  
\begin{align*}
&\|q_n(\hat h)\varphi \psi - \tilde q_n(\tilde h) \tilde \varphi \tilde \psi\|
\\
&
\lesssim
\biggl(\frac{M }{m\sqrt{n}}\|g_n''(0)-g''(0)\|+ \frac{M L}{m\sqrt{n}}(\|\hat h\|+\|\tilde h\|)\biggr)\biggl(\sum_{j=1}^n \|x_j-\tilde x_j\|^2\biggr)^{1/2}. 
\end{align*}
This implies the following bound on the local Lipschitz constant of $q_n(\hat h)\varphi \psi$ on set $A:$
\begin{align}
\label{bd_Lip_const_q_n_varphi_psi}
\nonumber
&
L(q_n(\hat h)\varphi \psi)(x_1,\dots, x_n)
\\
&
\lesssim
\biggl(\frac{M }{m\sqrt{n}}\Bigl(\|g_n''(0)-g''(0)\|\wedge \frac{m}{4}\Bigr)+ \frac{M L}{m\sqrt{n}}\Bigl(\|\hat h\|\wedge \frac{m}{12L}\Bigr)\biggr).
\end{align}
The same bound trivially holds on the open set $A^c$ (where $q_n(\hat h)\varphi \psi)(x_1,\dots, x_n)=0$)
and, by the argument already used at the end of the proof of Proposition \ref{Lip_check_theta}, it is easy to conclude 
that bound \eqref{bd_Lip_const_q_n_varphi_psi} holds on the whole space. 

Using bound \eqref{drei} and the bound of Proposition \ref{Lip_basic}, we get, for all $(x_1,\dots, x_n), (\tilde x_1,\dots, \tilde x_n)\in A,$
\begin{align*}
\|r(\hat h)-r(\tilde h)\| \leq \frac{2ML}{m \sqrt{n}}(\|\hat h\|+\|\tilde h\|)\biggl(\sum_{j=1}^n \|x_j-\tilde x_j\|^2\biggr)^{1/2}
\end{align*}
and we also have 
$
\|r(\hat h)\| \leq \frac{L}{2} \|\hat h\|^2.
$
As a result, we get the following condition for the function $r(\hat h) \varphi \psi$ on set $A:$
\begin{align*}
\|r(\hat h) \varphi \psi- r(\tilde h) \tilde \varphi \tilde \psi\|
\lesssim \frac{2ML}{m \sqrt{n}}(\|\hat h\|+\|\tilde h\|)\biggl(\sum_{j=1}^n \|x_j-\tilde x_j\|^2\biggr)^{1/2}.
\end{align*}
This implies a bound on the local Lipschitz constant of $r(\hat h) \varphi \psi$ on set $A$ that, by the arguments already used, could be extended 
to the bound that holds on the whole space:
\begin{align*} 
L(r(\hat h) \varphi \psi)(x_1,\dots, x_n) \lesssim \frac{ML}{m \sqrt{n}}\Bigl(\|\hat h\|\wedge \frac{m}{12L}\Bigr). 
\end{align*}

Denoting 
\begin{align*}
\zeta(x_1,\dots, x_n):=(- \langle q_n(\hat h), u\rangle \varphi \psi
+ \langle r(\hat h), u\rangle \varphi \psi)(x_1,\dots, x_n),
\end{align*}
we can conclude that 
\begin{align}
\label{local_lip_bd}
&
(L\zeta)(x_1,\dots, x_n)
\lesssim 
\|u\|\biggl(\frac{M }{m\sqrt{n}}\Bigl(\|g_n''(0)-g''(0)\|\wedge \frac{m}{4}\Bigr)+ \frac{M L}{m\sqrt{n}}\Bigl(\|\hat h\|\wedge \frac{m}{12L}\Bigr)\biggr).
\end{align}
By the second bound of Proposition \ref{bd_g_n''},
\begin{align*}
\Bigl\|\|g_n''(0)-g''(0)\|\Bigr\|_{\psi_2} 
\lesssim M \sqrt{\frac{d}{n}}.
\end{align*}
By the bound of Corollary \ref{bd_on_hat_theta_wedge_cor},
\begin{align}
\label{bd_on_hat_theta_wedge}
\Bigl\|\|\hat \theta-\theta\| \wedge \frac{m}{12L}\Bigr\|_{\psi_2}
\lesssim  \frac{\sqrt{M}}{m} \sqrt{\frac{d}{n}}+\frac{m}{L\sqrt{\gamma}}\frac{1}{\sqrt{n}}.
\end{align}
Substituting the above bounds in \eqref{local_lip_bd}, we conclude that 
\begin{align*}
\Bigl\|(L\zeta)(X_1,\dots, X_n)\Bigr\|_{\psi_2}
&\lesssim 
\Bigl(\frac{M^2}{m\sqrt{n}} 
\sqrt{\frac{d}{n}} + \frac{M^{3/2} L}{m^2\sqrt{n}}\sqrt{\frac{d}{n}}+
\frac{M}{\sqrt{\gamma}}\frac{1}{n}\Bigr)\|u\|
\\
&
\lesssim \Bigl(\frac{M^2}{m\sqrt{n}} 
\sqrt{\frac{d}{n}} + \frac{M^{3/2} L}{m^2\sqrt{n}}\sqrt{\frac{d}{n}}\Bigr)\|u\|,
\end{align*}
where we also used the fact that the term $\frac{M}{\sqrt{\gamma}}\frac{1}{n}$ is dominated by other terms.

We are now ready to use concentration inequalities for functions of log-concave r.v. to control $\zeta (X_1,\dots, X_n)- {\mathbb E} \zeta (X_1,\dots, X_n).$ For all $p\geq 1,$ we have 
\begin{align*}
&
\Bigl\|\zeta (X_1,\dots, X_n)- {\mathbb E} \zeta (X_1,\dots, X_n)\Bigr\|_{L_p} 
\\
&
\lesssim \sqrt{c(V)} p \Bigl\|(L\zeta)(X_1,\dots, X_n)\Bigr\|_{L_p}
\lesssim \sqrt{c(V)} p^{3/2} \Bigl\|(L\zeta)(X_1,\dots, X_n)\Bigr\|_{\psi_2}.
\end{align*}
It follows that 
\begin{align*}
\Bigl\|\zeta (X_1,\dots, X_n)- {\mathbb E} \zeta (X_1,\dots, X_n)\Bigr\|_{\psi_{2/3}}
&\lesssim  \sqrt{c(V)} \Bigl\|(L\zeta)(X_1,\dots, X_n)\Bigr\|_{\psi_2}
\\
&
\lesssim \sqrt{c(V)} \Bigl(\frac{M^2}{m\sqrt{n}} 
\sqrt{\frac{d}{n}} + \frac{M^{3/2} L}{m^2\sqrt{n}}\sqrt{\frac{d}{n}}\Bigr)\|u\|.
\end{align*}

Recalling representation \eqref{repr_lin_form_check_theta} and bound \eqref{odin}, we get 
\begin{align*}
&
\Bigl\|\langle \check \theta - \theta, w\rangle - {\mathbb E}\langle \check \theta - \theta, w\rangle - \langle {\mathcal I}^{-1}g_n'(0), w\rangle\Bigr\|_{\psi_{2/3}}
\\
&
\lesssim  \sqrt{c(V)} \Bigl(\frac{M^2}{m\sqrt{n}} 
\sqrt{\frac{d}{n}} + \frac{M^{3/2} L}{m^2\sqrt{n}}\sqrt{\frac{d}{n}}\Bigr)\|u\|
+\sqrt{\frac{M}{\gamma}}\frac{1}{n}\|u\|. 
\end{align*}
Since $c(V)\geq \|\Sigma\|\geq \|\mathcal{I}^{-1}\|$ and $M\geq \|\mathcal{I}\|$, we have $c(V)M\geq 1.$ 
Recalling the definition of $\gamma$ and also that 
$\|u\|= \|{\mathcal I}^{-1}w\|\leq \frac{1}{m} \|w\|,$ it is easy to complete the proof.

\qed
\end{proof}

We are ready to prove Theorem \ref{smooth_func_conc}.

\begin{proof}
Note that 
\begin{align*}
f(\check \theta) - f(\theta) = \langle f'(\theta), \check \theta -\theta \rangle + S_f(\theta;\check \theta-\theta).
\end{align*}
Therefore, 
\begin{align}
\label{rep_odin}
\nonumber
&
f(\check \theta) - {\mathbb E}_{\theta} f(\check \theta)
\\
&
=\langle f'(\theta), \check \theta -\theta \rangle - {\mathbb E}_{\theta}\langle f'(\theta), \check \theta-\theta \rangle 
+ S_f(\theta;\check \theta-\theta)- {\mathbb E}_{\theta}S_f(\theta;\check \theta-\theta).
\end{align}
By the bound of Theorem \ref{lin_func_conc}, 
\begin{align}
\label{bd_dva}
\nonumber
&
\Bigl\|\langle f'(\theta), \check \theta - \theta \rangle -{\mathbb E}\langle f'(\theta),\check \theta - \theta\rangle - n^{-1}\sum_{j=1}^n \langle V'(\xi_j), {\mathcal I}^{-1} f'(\theta)\rangle\Bigr\|_{\psi_{2/3}}
\\
&
\leq \sqrt{c(V)} \Bigl(\frac{M^2}{m^2} 
 + \frac{M^{3/2} L}{m^3}\Bigr) \frac{1}{\sqrt{n}}\sqrt{\frac{d}{n}}\|f'(\theta)\|.
\end{align}
Thus, it remains to control $S_f(\theta;\check \theta-\theta)- {\mathbb E}_{\theta}S_f(\theta;\check \theta-\theta).$
By Proposition \ref{remainder_bd} (v),
for function $f\in C^s,$ $s=1+\rho,$ $\rho\in (0,1],$ we have 
\begin{align*}
|S_f(\theta; h) - S_f(\theta; h')| \lesssim \|f\|_{C^s} (\|h\|^{\rho} \vee \|h'\|^{\rho}) \|h-h'\|, \theta, h,h'\in {\mathbb R}^d.
\end{align*}
Combining this with the bound of Proposition \ref{Lip_check_theta}, we easily get 
\begin{align*}
&
\Bigl|S_f(\theta; \check \theta (x_1,\dots, x_n)-\theta) - S_f(\theta; \check \theta (\tilde x_1,\dots, \tilde x_n)-\theta)\Bigr| 
\\
&
\lesssim \|f\|_{C^s} \frac{M}{m \sqrt{n}}(\|\check \theta (x_1,\dots, x_n)-\theta\|^{\rho}\vee \|\check \theta (\tilde x_1,\dots, \tilde x_n)-\theta\|^{\rho}) \Bigl(\sum_{j=1}^n \|x_j-\tilde x_j\|^2\Bigr)^{1/2},
\end{align*}
which implies the following bound on the local Lipschitz function of function $S_f(\theta; \check \theta -\theta):$
\begin{align*}
(L S_f(\theta; \check \theta -\theta))(x_1,\dots, x_n)
\lesssim \|f\|_{C^s} \frac{M}{m \sqrt{n}}\|\check \theta (x_1,\dots, x_n)-\theta\|^{\rho}.
\end{align*}
Using concentration bounds for log-concave r.v., we get 
\begin{align*}
&
\Bigl\|S_f(\theta;\check \theta-\theta)- {\mathbb E}_{\theta}S_f(\theta;\check \theta-\theta)\Bigr\|_{L_p}
\lesssim \sqrt{c(V)} p \Bigl\|(L S_f(\theta; \check \theta -\theta))(X_1,\dots, X_n)\Bigr\|_{L_p}
\\
&
\lesssim \sqrt{c(V)} p  \|f\|_{C^s} \frac{M}{m \sqrt{n}} \Bigl\|\|\check \theta -\theta\|^{\rho}\Bigr\|_{L_p}
\lesssim   \sqrt{c(V)} p  \|f\|_{C^s} \frac{M}{m \sqrt{n}} \Bigl\|\|\check \theta -\theta\|\Bigr\|_{L_p}^{\rho}
\\
&
\lesssim \sqrt{c(V)} p  \|f\|_{C^s} \frac{M}{m \sqrt{n}} \Bigl\|\|\hat \theta -\theta\|\wedge \frac{m}{12L}\Bigr\|_{L_p}^{\rho}
\lesssim \sqrt{c(V)} p^{1+\rho/2}  \|f\|_{C^s} \frac{M}{m \sqrt{n}} \Bigl\|\|\hat \theta -\theta\|\wedge \frac{m}{12L}\Bigr\|_{\psi_2}^{\rho},
\end{align*}
which, using bound \eqref{bd_on_hat_theta_wedge}, implies that  
\begin{align*}
&
\Bigl\|S_f(\theta;\check \theta-\theta)- {\mathbb E}_{\theta}S_f(\theta;\check \theta-\theta)\Bigr\|_{\psi_{2/(2+\rho)}}
\lesssim 
\sqrt{c(V)}  \|f\|_{C^s} \frac{M}{m \sqrt{n}} \Bigl\|\|\hat \theta -\theta\|\wedge \frac{m}{12L}\Bigr\|_{\psi_2}^{\rho}
\\
&
\lesssim \sqrt{c(V)}  \|f\|_{C^s} \biggl(\frac{M^{1+\rho/2}}{m^{1+\rho}}\frac{1}{\sqrt{n}}\Bigl(\frac{d}{n}\Bigr)^{\rho/2} +  
\frac{M}{L^{\rho} m^{1-\rho} \gamma^{\rho/2}}\frac{1}{n^{(1+\rho)/2}}\biggr).
\end{align*}
Combining this with \eqref{rep_odin} and \eqref{bd_dva}, we get 
\begin{align*}
&
\Bigl\|f(\check \theta) - {\mathbb E}_{\theta} f(\check \theta)- n^{-1}\sum_{j=1}^n \langle V'(\xi_j), {\mathcal I}^{-1} f'(\theta)\rangle\Bigr\|_{\psi_{2/3}}
\\
&
\lesssim 
\sqrt{c(V)} \|f'(\theta)\|\Bigl(\frac{M^2}{m^2} + \frac{M^{3/2} L}{m^3} \Bigr)\frac{1}{\sqrt{n}}\sqrt{\frac{d}{n}}
\\
&
+ \sqrt{c(V)}  \|f\|_{C^s} \biggl(\frac{M^{1+\rho/2}}{m^{1+\rho}}\frac{1}{\sqrt{n}}\Bigl(\frac{d}{n}\Bigr)^{\rho/2} +  
\frac{M}{L^{\rho} m^{1-\rho} \gamma^{\rho/2}}\frac{1}{n^{(1+\rho)/2}}\biggr)
\\
&
\lesssim_{M,L,m} \sqrt{c(V)} \|f\|_{C^s} \frac{1}{\sqrt{n}}\Bigl(\frac{d}{n}\Bigr)^{\rho/2}.
\end{align*}
It remains to replace in the above bound $\check \theta$ by $\hat \theta.$ To this end, observe that 
$|f(\hat \theta)-f(\check \theta)|\leq 2\|f\|_{L_{\infty}} I(\check\theta\neq \hat \theta).$
This implies 
\begin{align*}
\|f(\hat \theta)-f(\check \theta)\|_{\psi_1} \leq 2\|f\|_{L_{\infty}} 
\|I(\check\theta\neq \hat \theta)\|_{\psi_1}
\lesssim \frac{\|f\|_{L_{\infty}}}{\gamma n} \lesssim \frac{\|f\|_{C^s}}{\gamma n},
\end{align*}
which allows us to complete the proof.

\qed
\end{proof}

\section{Bias reduction}
We turn now to the bias reduction method outlined in Section \ref{intro}.
The justification of this method is much simpler in the case of equivariant 
estimators $\hat \theta$ of location parameter. 
Indeed, in this case 
\begin{align*}
\hat \theta (X_1,\dots, X_n) = \theta + \hat \theta (\xi_1,\dots, \xi_n), 
\end{align*}
where $\xi_j= X_j-\theta, j=1,\dots, n$ are i.i.d. $\sim P,$ $P(dx)=e^{-V(x)}dx.$
Denote $\vartheta:=\hat \theta (\xi_1,\dots, \xi_n)$ and let $\{\vartheta_k\}$ be a sequence 
of i.i.d. copies of $\vartheta$ defined as follows: 
$
\vartheta_k := \hat \theta (\xi_1^{(k)},\dots, \xi_n^{(k)}),
$
$\xi_j^{(k)}, j=1,\dots, n, k\geq 1$ being i.i.d. copies of $\xi.$
Then, the bootstrap chain $\{\hat \theta^{(k)}: k\geq 0\}$
has the same distribution as the sequence of r.v.
$\{\theta + \sum_{j=1}^k \vartheta_j: k\geq 0\}.$
Moreover, let 
\begin{align*}
\vartheta(t_1,\dots, t_k) := \sum_{j=1}^k t_j\vartheta_j,\  (t_1,\dots, t_k)\in [0,1]^k.
\end{align*}
Then, for $(t_1,\dots, t_k)\in \{0,1\}$ with $\sum_{i=1}^n t_i=j,$ we have 
$\theta+\vartheta(t_1,\dots, t_k)\overset{d}= \hat \theta^{(j)}.$ Therefore, 
we can write 
\begin{align*}
({\mathcal B}^k f)(\theta)&= {\mathbb E}_{\theta}\sum_{j=0}^{k} (-1)^{k-j}\binom{k}{j} f(\hat \theta^{(j)})
\\
&
= {\mathbb E} \sum_{(t_1,\dots, t_k)\in \{0,1\}^k} (-1)^{k-\sum_{j=1}^k t_j} f(\theta+\vartheta(t_1,\dots, t_k))
\\
&
= {\mathbb E} \Delta_1 \dots \Delta_k f(\theta+\vartheta(t_1,\dots, t_k)),
\end{align*}
where 
\begin{align*}
\Delta_j \varphi (t_1,\dots, t_k):=  \varphi (t_1,\dots, t_k)_{| t_j=1}- \varphi (t_1,\dots, t_k)_{| t_j=0}.
\end{align*}
If function $\varphi$ is $k$ times continuously differentiable, then by Newton-Leibniz formula
\begin{align*}
\Delta_1 \dots \Delta_k \varphi (t_1,\dots, t_k)= \int_0^1 \dots \int_0^1 \frac{\partial^k \varphi (t_1,\dots, t_k)}{\partial t_1\dots \partial t_k} d t_1\dots d t_k.
\end{align*}
If $f :{\mathbb R}^d \mapsto {\mathbb R}$ is $k$ times continuously differentiable, then 
\begin{align*}
\frac{\partial^k}{\partial t_1\dots \partial t_k} f(\theta+\vartheta(t_1,\dots, t_k))= f^{(k)}(\theta+\vartheta(t_1,\dots, t_k))[\vartheta_1,\dots, \vartheta_k] 
\end{align*} 
and we end up with the following integral representation formula 
\begin{align}
\label{repr_form_B^k}
({\mathcal B}^k f)(\theta)= {\mathbb E} \int_0^1 \dots \int_0^1
f^{(k)}(\theta+\vartheta(t_1,\dots, t_k))[\vartheta_1,\dots, \vartheta_k] 
d t_1\dots d t_k
\end{align}
that plays an important role in the analysis of functions ${\mathcal B}^k f$ and $f_k.$

It will be convenient to apply this formula not directly to MLE $\hat \theta,$ but to its smoothed and truncated 
approximation $\check \theta,$ defined in the previous section. 
Note that $\check \theta (X_1,\dots, X_n)=\theta + \check \vartheta,$
where 
\begin{align*}
\check \vartheta &= \check \vartheta(\xi_1,\dots, \xi_n)
\\
&
:= \hat \theta (\xi_1,\dots, \xi_n) 
\phi\Bigl(\frac{24L}{m}\|\hat \theta(\xi_1,\dots, \xi_n)\|\Bigr) 
\phi\Bigl(\frac{8}{m}\|g_n''(0)(\xi_1,\dots, \xi_n)-g''(0)\|\Bigr). 
\end{align*} 
Let $\check \vartheta_k, k\geq 1$ be i.i.d. copies of $\check \vartheta$ defined as follows:
\begin{align*}
\check \vartheta_k 
:= \vartheta_k
\phi\Bigl(\frac{24L}{m}\|\vartheta_k\|\Bigr) 
\phi\Bigl(\frac{8}{m}\|g_n''(0)(\xi_1^{(k)},\dots, \xi_n^{(k)})-g''(0)\|\Bigr). 
\end{align*} 
Note that, for all $k\geq 1,$ $\vartheta_k= \check \vartheta_k$ with probability at least $1-2e^{-\gamma n}.$

Let $\check \theta^{(k)}:= \theta + \sum_{j=1}^k \check \vartheta_j, k\geq 0.$
We will also introduce the operators $(\check {\mathcal T} g)(\theta):= {\mathbb E}_{\theta} g(\check \theta), \theta\in {\mathbb R}^d$ and 
$\check {\mathcal B}:= \check {\mathcal T}-{\mathcal I}.$
Let $\check f_k:= \sum_{j=0}^k (-1)^j \check {\mathcal B}^j f.$
Since $\vartheta_j=\check \vartheta_j, j=1,\dots, k$ with probability at least $1-2 k e^{-\gamma n},$ we easily conclude 
that, if we identify $\hat\theta^{(k)}$ with $\theta + \sum_{j=1}^k \vartheta_j$, $k\geq 0$, then the event $E:=\{\check \theta^{(j)}=\hat \theta^{(j)}, j=1,\dots, k\}$ occurs with the same probability. This immediately implies the following 
proposition.

\begin{proposition}
\label{f_k-check_f_k}
For all $k\geq 1,$
\begin{align}
\|f_k-\check f_k\|_{L_{\infty}}
\leq k2^{k+3}\|f\|_{L_{\infty}} e^{-\gamma n}
\end{align}
and 
\begin{align}
\label{bounded_f_k}
\|f_k\|_{L_{\infty}} \leq 2^{k+1}\|f\|_{L_{\infty}},\qquad \|\check f_k\|_{L_{\infty}}\leq  2^{k+1}\|f\|_{L_{\infty}}.
\end{align}
\end{proposition}

\begin{proof}
For all $\theta\in {\mathbb R}^d$ and all $j=1,\dots, k,$ 
\begin{align*}
|{\mathbb E}_{\theta} f(\hat \theta^{(j)})- {\mathbb E}_{\theta} f(\check \theta^{(j)})|&=|{\mathbb E}_{\theta} f(\hat \theta^{(j)})I_{E^c}- {\mathbb E}_{\theta} f(\check \theta^{(j)}) I_{E^c}|
\\
&
\leq 2\|f\|_{L_{\infty}} {\mathbb P}(E^c) \leq 4 \|f\|_{L_{\infty}} k e^{-\gamma n}.
\end{align*}
Therefore, applying \eqref{representation_f_k_alternative} to $f_k$ and $\check f_k$, we arrive at
\begin{align*}
|f_k(\theta)-\check f_k(\theta)|&\leq 
\sum_{j=0}^k \binom{k+1}{j+1}|{\mathbb E}_{\theta} f(\hat \theta^{(j)})- {\mathbb E}_{\theta} f(\check \theta^{(j)})|
\\
&
\leq k2^{k+3}\|f\|_{L_{\infty}} e^{-\gamma n},
\end{align*}
which proofs Proposition \eqref{f_k-check_f_k}.
Bounds \eqref{bounded_f_k} follow by a similar argument.


\qed
\end{proof}

Similarly to \eqref{repr_form_B^k}, we get
\begin{align}
\label{repr_check_B^k}
\nonumber
(\check {\mathcal B}^k f)(\theta)
&= {\mathbb E} \int_0^1 \dots \int_0^1
f^{(k)}(\theta+\check \vartheta(t_1,\dots, t_k))[\check \vartheta_1,\dots, \check \vartheta_k] 
d t_1\dots d t_k
\\
&
= 
{\mathbb E} 
f^{(k)}(\theta+\check \vartheta(\tau_1,\dots, \tau_k))[\check \vartheta_1,\dots, \check \vartheta_k],
\end{align}
where $\check \vartheta(t_1,\dots, t_k) := \sum_{j=1}^k t_j\check \vartheta_j,\  (t_1,\dots, t_k)\in [0,1]^k$
and $\tau_1,\dots, \tau_k$ are i.i.d. r.v. with uniform distribution in $[0,1]$ (independent of $\{\check \vartheta_j\}$).

The next proposition follows from representation \eqref{repr_check_B^k} and differentiation under the expectation sign.

\begin{proposition}
\label{norms_of_check_B^k}
Let $f\in C^s$ for $s=k+1+\rho,$ where $k\geq 1$ and $\rho\in (0,1].$
Then, for all $j=1,\dots, k,$ 
\begin{align*}
\|\check {\mathcal B}^j f\|_{C^{1+\rho}} \lesssim \|f\|_{C^s} ({\mathbb E}\|\check \vartheta\|)^j.
\end{align*}
If ${\mathbb E}\|\check \vartheta\|\leq 1/2,$ then 
\begin{align*}
\|\check f_k\|_{C^{1+\rho}} \lesssim \|f\|_{C^s}.
\end{align*}
\end{proposition}

We can also use representation \eqref{repr_check_B^k} and smoothness of function $\check {\mathcal B}^k f$
to obtain a bound on the bias of ``estimator" $\check f_k(\check \theta).$ 

\begin{proposition}
\label{bias_of_check_f_k}
Let $f\in C^s$ for $s=k+1+\rho,$ where $k\geq 1$ and $\rho\in (0,1].$
Then, for all $\theta\in {\mathbb R}^d,$ 
\begin{align*}
|{\mathbb E}_{\theta} \check f_k(\check \theta) -f(\theta)|
\lesssim 
\|f\|_{C^s} ({\mathbb E}\|\check \vartheta\|)^k (\|{\mathbb E} \check \vartheta\|+  {\mathbb E}\|\check \vartheta\|^{1+\rho}).
\end{align*}
Moreover,
\begin{align*}
|{\mathbb E}_{\theta} \check f_k(\check \theta) -f(\theta)| \lesssim_{M,L,m} \|f\|_{C^s}\Bigl(\sqrt{\frac{d}{n}}\Bigr)^s.
\end{align*}
\end{proposition}

\begin{proof}
Note that 
\begin{align*}
{\mathbb E}_{\theta} \check f_k(\check \theta) -f(\theta)= (-1)^k (\check {\mathcal B}^{k+1} f)(\theta).
\end{align*}
We also have 
\begin{align*}
(\check {\mathcal B}^{k+1} f)(\theta)&= {\mathbb E}_{\theta} (\check {\mathcal B}^{k} f)(\check \theta) - (\check {\mathcal B}^{k} f)(\theta)
\\
&
=\langle (\check {\mathcal B}^{k} f)'(\theta), {\mathbb E}\check \vartheta\rangle +  
{\mathbb E} S_{\check {\mathcal B}^{k} f}(\theta ; \check \vartheta).
\end{align*}
Using bounds of Proposition \ref {norms_of_check_B^k}
and of Proposition \ref{remainder_bd}, we get 
\begin{align*}
|(\check {\mathcal B}^{k+1} f)(\theta)| &\lesssim \|(\check {\mathcal B}^{k} f)'\| \|{\mathbb E} \check \vartheta\|
+\|\check {\mathcal B}^{k} f\|_{C^{1+\rho}} {\mathbb E}\|\check \vartheta\|^{1+\rho} 
\\
&
\lesssim 
\|f\|_{C^s} ({\mathbb E}\|\check \vartheta\|)^k (\|{\mathbb E} \check \vartheta\|+  {\mathbb E}\|\check \vartheta\|^{1+\rho}).
\end{align*}
Using also Proposition \ref{bd_on_bias}, we get
\begin{align*}
|{\mathbb E}_{\theta} \check f_k(\check \theta) -f(\theta)|
\lesssim _{M,L,m}
\|f\|_{C^s} \Bigl(\frac{d}{n}\Bigr)^{k/2} \Bigl(\frac{d}{n}+  \Bigl(\frac{d}{n}\Bigr)^{(1+\rho)/2}\Bigr)
\end{align*}
which allows to complete the proof.

%
%

\qed
\end{proof}

In view of bound \eqref{bounded_f_k}, the bound of Proposition \ref{f_k-check_f_k} and the fact that $\check \theta= \hat \theta$ with 
probability at least $1-2e^{-\gamma n},$ we easily conclude that the following proposition holds:

\begin{proposition}
\label{bounds_on_bias}
Let $f\in C^s$ for $s=k+1+\rho,$ where $k\geq 1$ and $\rho\in (0,1].$ Then, for all $\theta\in {\mathbb R}^d,$ 
\begin{align*}
|{\mathbb E}_{\theta} \check f_k(\hat \theta) -f(\theta)|
\lesssim _{M,L,m,s}
\|f\|_{C^s}\Bigl(\sqrt{\frac{d}{n}}\Bigr)^s
\end{align*}
and 
\begin{align}
|{\mathbb E}_{\theta} f_k(\hat \theta) -f(\theta)|
\lesssim _{M,L,m,s}
\|f\|_{C^s}\Bigl(\sqrt{\frac{d}{n}}\Bigr)^s.
\end{align}
\end{proposition}

It is now easy to prove Theorem \ref{main_theorem}.

\begin{proof}
For all $\theta\in\mathbb{R}^d$, 
\begin{align}
& 
\nonumber
\Bigl\|f_k(\hat \theta) - f(\theta)- n^{-1}\sum_{j=1}^n \langle V'(\xi_j), {\mathcal I}^{-1} f'(\theta)\rangle\Bigr\|_{\psi_{2/3}}
\\
&
\nonumber
\leq \Bigl\|\check f_k(\hat \theta) - {\mathbb E}_{\theta} \check f_k(\hat \theta)- n^{-1}\sum_{j=1}^n \langle V'(\xi_j), {\mathcal I}^{-1} \check f_k'(\theta)\rangle\Bigr\|_{\psi_{2/3}}
\\
&
\nonumber
+\Bigl\| n^{-1}\sum_{j=1}^n \langle V'(\xi_j), {\mathcal I}^{-1}\check f'_k(\theta)-{\mathcal I}^{-1} f'(\theta)\rangle\Bigr\|_{\psi_{2/3}}
\\
&
\label{eq_decomposition_proof_main_theorem}
+\|f_k-\check f_k\|_{L_\infty}+|\mathbb{E}_\theta\check f_k(\hat \theta)-f(\theta)|.
\end{align}
Applying Theorem \ref{smooth_func_conc} to function $\check f_k$ and using the second bound of Proposition \ref{norms_of_check_B^k}, we get
\begin{align*}
&
\Bigl\|\check f_k(\hat \theta) - {\mathbb E}_{\theta} \check f_k(\hat \theta)- n^{-1}\sum_{j=1}^n \langle V'(\xi_j), {\mathcal I}^{-1} \check f_k'(\theta)\rangle\Bigr\|_{\psi_{2/3}}
\\
&
\lesssim_{M,L,m} \sqrt{c(V)} \|\check f_k\|_{C^{1+\rho}} \frac{1}{\sqrt{n}}\Bigl(\frac{d}{n}\Bigr)^{\rho/2}
\lesssim_{M,L,m} \sqrt{c(V)} \|f\|_{C^s} \frac{1}{\sqrt{n}}\Bigl(\frac{d}{n}\Bigr)^{\rho/2}.
\end{align*}
Moreover, by Lemma \ref{bd_on_V'} and the first bound in Proposition \ref{norms_of_check_B^k}, we have
\begin{align*}
   &\Bigl\| n^{-1}\sum_{j=1}^n \langle V'(\xi_j), {\mathcal I}^{-1}\check f'_k(\theta)-{\mathcal I}^{-1} f'(\theta)\rangle\Bigr\|_{\psi_{2/3}}\\
   &\lesssim\frac{1}{m}\frac{\sqrt{M}}{\sqrt{n}}\|\check f'_k(\theta)-f'(\theta)\|\leq  \frac{1}{m}\frac{\sqrt{M}}{\sqrt{n}}\sum_{j=1}^k\|(\mathcal{B}^jf)'(\theta)\|\lesssim \|f\|_{C^s}\frac{1}{m}\frac{\sqrt{M}}{\sqrt{n}}.
\end{align*}
Inserting these inequalities into \eqref{eq_decomposition_proof_main_theorem} and applying  Propositions \ref{f_k-check_f_k} and \ref{bias_of_check_f_k} to the last two term in \eqref{eq_decomposition_proof_main_theorem}, allows to complete the proof.


\qed
\end{proof}

Next we provide the proof of Proposition \ref{prop_min_max_upper}.

\begin{proof}
The minimum with $1$ in both bounds is due to the fact that $\|f_k\|_{L_{\infty}}\lesssim \|f\|_{L_{\infty}}\leq \|f\|_{C^s};$
so, the left hand side is trivially bounded up to a constant by $\|f\|_{C^s}.$

To prove the first claim, note that, for $f$ with $\|f\|_{C^s}\leq 1,$
\begin{align*}
\|f(\hat \theta)-f(\theta)\|_{L_2({\mathbb P}_{\theta})} &\leq 
\Bigl\|\Bigl(\|\hat \theta -\theta\|\wedge \frac{m}{12L}\Bigr)^s\Bigr\|_{L_2({\mathbb P}_{\theta})}
+ 2\Bigl\|I(\|\hat \theta-\theta\|\geq m/(12L))\Bigr\|_{L_2({\mathbb P}_{\theta})}
\\
&
\leq \Bigl(\Bigl\|\|\hat \theta -\theta\|\wedge \frac{m}{12L}\Bigr\|_{L_2({\mathbb P}_{\theta})}\Bigr)^s
+ 2 {\mathbb P}_{\theta}^{1/2}\Bigl\{\|\hat \theta-\theta\|\geq m/(12L)\Bigr\}.
\end{align*}
Using the bound of Corollary \ref{bd_on_hat_theta_wedge_cor}, we get 
\begin{align*}
\Bigl(\Bigl\|\|\hat \theta -\theta\|\wedge \frac{m}{12L}\Bigr\|_{L_2({\mathbb P}_{\theta})}\Bigr)^s
\lesssim_{M,L,m} \Bigl(\sqrt{\frac{d}{n}}\Big)^s,
\end{align*}
and, by the bound of Theorem \ref{th_bound_on_hat_theta}, we easily get 
\begin{align*}
{\mathbb P}_{\theta}\Bigl\{\|\hat \theta-\theta\|\geq m/(12L)\Bigr\} \leq e^{-\gamma n}.
\end{align*}
The first claim now easily follows.

The proof of the second claim easily follows from Corollary \ref{cor_L_2}.
We can assume that $d\leq \gamma n$ (otherwise, the bound is obvious)
and we can drop the term $\sqrt{\frac{c(V)}{n}}\Bigl(\frac{d}{n}\Bigr)^{\rho/2}$ in the bound of Corollary \ref{cor_L_2}: it is smaller than $\frac{1}{\sqrt{n}}$ since $c(V)\lesssim_{\epsilon} d^{\epsilon}\|\Sigma\|$ for all $\epsilon>0$ and $\|\Sigma\|\lesssim 1.$

\qed
\end{proof}

We will sketch the proof of Theorem  \ref{main_theorem_AAA}.

\begin{proof}
Under the stronger condition $V''(x)\succeq m I_d,$ the proof of Theorem \ref{main_theorem} could be significantly simplified.
Recall that, for $\hat h=\theta-\hat \theta,$ $g_n'(\hat h)=0.$ This implies that 
\begin{align}
\label{od}
g_n'(0) = g_n'(0)-g_n'(\hat h) = - \int_0^1 g_n''(\lambda \hat h) d\lambda\ \hat h.
\end{align}
The condition $V''(x)\succeq m I_d$ easily implies that 
\begin{align*}
\int_0^1 g_n''(\lambda \hat h) d\lambda=\int_0^1 n^{-1}\sum_{j=1}^n V''(\xi_j + \lambda \hat h) d\lambda 
\succeq m I_d.
\end{align*}
Therefore, $\Bigl\|\int_0^1 g_n''(\lambda \hat h) d\lambda\ u\Bigr\|\geq m \|u\|, u\in {\mathbb R}^d.$
Combining this with \eqref{od} yields
$
\|\hat h\| \leq \frac{\|g_n'(0)\|}{m}.
$
By Corollary \ref{cor_g_n'}, we get that 
for all $t\geq 1,$ with probability at least $1-e^{-t}$
\begin{align*}
\|\hat \theta-\theta\|= \|\hat h\| \lesssim \frac{\sqrt{M}}{m}\Bigl(\sqrt{\frac{d}{n}}\vee \sqrt{\frac{t}{n}}\Bigr).
\end{align*}
Unlike the case of Theorem \ref{th_bound_on_hat_theta}, the above bound holds in the whole range of $t\geq 1$
and it immediately implies that ${\mathbb E}_{\theta}^{1/2}\|\hat \theta -\theta\|^2\lesssim \frac{\sqrt{M}}{m}\sqrt{\frac{d}{n}},$
and, moreover, $\|\|\hat \theta -\theta\|\|_{L_{\psi_2}({\mathbb P}_{\theta})}\lesssim \frac{\sqrt{M}}{m}\sqrt{\frac{d}{n}}.$

Quite similarly, one can show that, unlike the case of Proposition \ref{Lip_basic},  the Lipschitz condition for the function 
${\mathbb R}^d\times \dots \times {\mathbb R}^d\ni(x_1,\dots, x_n) \mapsto \hat \theta(x_1,\dots, x_n)\in {\mathbb R}^d$
holds not just on set $A,$ but on the whole space. Indeed, recall that $g_n'(\hat h)=0$ and $\tilde g_n'(\tilde h)=0.$
This implies that 
\begin{align*}
\tilde g_n'(\tilde h) - g_n'(\tilde h) = g_n'(\hat h) - g_n'(\tilde h)
=\int_{0}^1 g_n''(\tilde h+ \lambda (\hat h-\tilde h))d\lambda\ (\hat h-\tilde h). 
\end{align*}
Since $\int_{0}^1 g_n''(\tilde h+ \lambda (\hat h-\tilde h))d\lambda\succeq m I_d,$ we get 
\begin{align*}
\|\tilde g_n'(\tilde h) - g_n'(\tilde h)\| \geq m \|\hat h -\tilde h\|,
\end{align*}
which implies 
\begin{align*}
\|\tilde \theta  -\hat \theta\|&= \|\hat h-\tilde h\| \leq m^{-1} n^{-1}\sum_{j=1}^n \|V'(\tilde h+ \tilde \xi_j)- V'(\tilde h+ \xi_j)\|
\\
&
\leq m^{-1} n^{-1} M \sum_{j=1}^n \|\tilde \xi_j-\xi_j\| \leq \frac{M}{m \sqrt{n}} \Bigl(\sum_{j=1}^n \|\tilde \xi_j-\xi_j\|^2\Bigr)^{1/2} 
\\
&
=\frac{M}{m \sqrt{n}} \Bigl(\sum_{j=1}^n \|\tilde x_j-x_j\|^2\Bigr)^{1/2},
\end{align*}
and the Lipschitz condition holds for the function 
${\mathbb R}^d\times \dots \times {\mathbb R}^d\ni(x_1,\dots, x_n) \mapsto \hat \theta(x_1,\dots, x_n)\in {\mathbb R}^d$
with constant $ \frac{M}{m \sqrt{n}} .$ 
Due to this fact, there is no need to consider a ``smoothed version" $\check \theta$ of function $\hat \theta$ in the remainder
of the proof (as it was done in the proof of Theorem \ref{main_theorem}).  All the arguments could be applied directly to $\hat \theta.$    

Finally, recall that, if $\xi \sim P,$ $P(dx)=e^{-V(x)}dx$ with $V''(x)\succeq m I_d, x\in {\mathbb R}^d,$
then, for all locally Lipschitz functions $g:{\mathbb R}^d\mapsto {\mathbb R},$ the following logarirthmic Sobolev inequality holds
\begin{align*}
{\mathbb E} g^2(\xi) \log g^2(\xi) - {\mathbb E}g^2(\xi) \log {\mathbb E}g^2(\xi)
\lesssim \frac{1}{m} {\mathbb E} \|\nabla g(\xi)\|^2
\end{align*}
 (see, e.g., \cite{Ledoux}, Theorem 5.2). It was proved in \cite{A_S} (see also \cite{ABW}) that this implies the following 
 moment bound:
 \begin{align*}
 \|g(\xi)-{\mathbb E}g(\xi)\|_{L_p} \lesssim_m \sqrt{p} \|\|\nabla g(\xi)\|\|_{L_p}, p\geq 2.
 \end{align*} 
 This bound is used to modify the concentration inequalities of Section \ref{Sec:conc}, which 
 yields the claim of Theorem \ref{main_theorem_AAA}. 

\qed
\end{proof}

\section{Minimax lower bounds}
\label{Sec:lower_bounds}
In this section, we provide lower bounds for the estimation of the location parameter and functionals thereof
that match the upper bounds obtained in the previous sections up to constants.

We start with a comment on the proof of Proposition \ref{min_max_lower}. 

\begin{proof}
It follows the same line of arguments as in the proof of Theorem 2.2 in \cite{Koltchinskii_Zhilova_19}. It is based on a construction of a set ${\mathcal F}$ of smooth functionals such that the existence of estimators 
of $f(\theta)$ for all $f\in {\mathcal F}$ with some error rate would allow one to design an estimator of parameter $\theta$ itself with a certain 
error rate. This rate is then compared with a minimax lower bound 
$\inf_{\hat \theta}\max_{\theta\in\Theta}\mathbb{E}_\theta\|\hat\theta-\theta\|^2$ in the parameter estimation problem, 
where $\Theta$ is a maximal $\eps$-net of the unit sphere (for a suitable $\eps$), yielding as a result a minimax lower bound in the 
functional estimation. Minimax lower bound in the parameter estimation problem can be deduced in a standard way from Theorem 2.5 of \cite{Tsybakov} using KL-divergence (Fano's type argument). 
In fact, while in the Gaussian location model KL-divergence coincides with $1/2$ times the squared Euclidean distance, a similar property also holds for our log-concave location models:
\begin{align*}
K(P_\theta\|P_{\theta'})
&=\mathbb{E} (V(\xi+\theta-\theta')-V(\xi))\\
&=\mathbb{E} (V(\xi+\theta-\theta')-V(\xi)-\langle V'(\xi),\theta-\theta'\rangle)
\leq M\|\theta-\theta'\|^2/2,
\end{align*}  
where we used Proposition \ref{remainder_bd} and Assumption \ref{Assump_main} in the inequality.

\qed
\end{proof}

Our next goal is to provide the proof of the local minimax lower bound of 
Proposition \ref{min_max_lower_local}.
It will be based on Bayes risk lower bounds for the estimation of location parameter as well as functionals thereof that might be of independent interest.

Let $(\mathcal{X},\mathcal{F},(\mathbb{P}_\theta)_{\theta\in\Theta})$ be a statistical model and let $G$ be a topological group acting on 
the measurable space $(\mathcal{X},\mathcal{F})$ and $\Theta.$ Let $\Theta$ be such that $(\mathbb{P}_\theta)_{\theta\in\Theta}$ is $G$-equivariant, 
i.e., $\mathbb{P}_{g\theta}(gA)=\mathbb{P}_\theta(A)$ for all $g\in G, \theta\in\Theta$ and $A\in\mathcal{F}.$ 
Suppose that $g\mapsto \mathbb{P}_{g\theta}(A)$ is measurable for every $A\in\mathcal{F}$, $\theta\in\Theta$.

Recall that, for two probability measures $\mu, \nu$ on an arbitrary measurable space with $\mu$ being absolutely continuous w.r.t. 
$\nu,$ the $\chi^2$-divergence $\chi^2(\mu,\nu)$ is defined as 
\begin{align*}
\chi^2(\mu,\nu):= \int \Bigl(\frac{d\mu}{d\nu}-1\Bigr)^2 d\nu = \int \Bigl(\frac{d\mu}{d\nu}\Bigr)^2d\nu-1.
\end{align*}

The key ingredient in our proofs is the following version of an equivariant van Trees type inequality established in \cite{W20}, Proposition 1.

\begin{lemma}
\label{prop:equivariant:chapman:robbins:ineq}
Let $\Pi$ be a Borel probability measure on $G,$ let $\psi:\Theta\rightarrow \mathbb{R}^m$ be a derived parameter such that $\int_G\|\psi(g\theta)\|^2\Pi(dg)<\infty$ for all $\theta\in\Theta$ and let  
$\hat\psi:\mathcal{X}\rightarrow \mathbb{R}^m$ be an estimator of $\psi(\theta),$ 
based on an observation 
$X\sim {\mathbb P}_{\theta}, \theta \in \Theta.$ 
Then, for all $\theta\in\Theta$ and all $h_1,\dots,h_m\in G$, we have  
\begin{align*}
&\int_{G}\mathbb{E}_{g\theta}\|\hat\psi(X)-\psi(g\theta)\|^2\,\Pi(dg)\\
&\geq \frac{\big( \sum_{j=1}^m\int_G(\psi_j(gh_j^{-1}\theta)-\psi_j(g\theta))\,\Pi(dg)\big)^2}{\sum_{j=1}^m\big(\chi^2(\mathbb{P}_{h_j\theta},\mathbb{P}_\theta)+\chi^2(\Pi\circ R_{h_j},\Pi)+\chi^2(\mathbb{P}_{h_j\theta},\mathbb{P}_\theta)\chi^2(\Pi\circ R_{h_j},\Pi)\big)},
\end{align*}
with $\Pi\circ R_{h_j}$ defined by $(\Pi\circ R_{h_j})(B)=\Pi(Bh_j)$ for all Borel sets $B\subset G.$
\end{lemma}

This lemma will be applied to our log-concave location model with $G$ being the group of all translations of ${\mathbb R}^d.$
The Bayes risk lower bound will be formulated for the class of all prior density functions $\pi:\mathbb{R}^d\mapsto {\mathbb R}_+$ with respect to the Lebesgue measure on $\mathbb{R}^d$ satisfying one of the following two conditions:
\begin{enumerate}
\item[(P1)] $\pi=e^{-W}$ with $W:\mathbb{R}^d\mapsto \mathbb{R}$ being twice differentiable such that $\|W''\|_{L_{\infty}}$ and $\|W''\|_{\operatorname{Lip}}$ are finite,
\item[(P2)] $\pi$ has bounded support and is twice differentiable such that $\|\pi^{(j)}\|_{L_{\infty}}$ and $\|\pi^{(j)}\|_{\operatorname{Lip}}$ are finite for $j=0,1,2$ (actually, 
it suffices to assume it only for $j=2$ since the support is bounded). 
\end{enumerate}

Our first result deals with the estimation of the location parameter itself. We assume that i.i.d. observations $X_1,\dots, X_n$
are sampled from a distribution belonging to a log-concave location family $e^{-V(x-\theta)}dx, \theta\in {\mathbb R}^d$
with convex function $V$ satisfying Assumption \ref{Assump_main}.

\begin{theorem}
\label{theorem:van:Trees:theta}
Let $\Pi$ be a probability measure on $\mathbb{R}^d$ with density  $\pi:\mathbb{R}^d\mapsto {\mathbb R}_+$ with respect to the Lebesgue measure satisfying either (P1) or (P2). Suppose that $\pi$ has finite Fisher information matrix
\begin{align*}
\mathcal{J}_\pi=\int_{\mathbb{R}^d}\frac{\pi'(\theta)\otimes \pi'(\theta)}{\pi(\theta)}\,d\theta.
\end{align*}
Then, for all $\delta>0$, we have
\begin{align}\label{eq:van:Trees:theta}
 \inf_{\hat\theta_n}\int_{\mathbb{R}^d}\mathbb{E}_{\theta}\|\hat\theta_n-\theta\|^2\,\Pi_\delta(d\theta)\geq 
 \frac{1}{n}\operatorname{tr}\Big(\Big(\mathcal{I}+\frac{1}{\delta^2n}\mathcal{J}_\pi\Big)^{-1}\Big),
\end{align}
where $\Pi_\delta$ is the prior measure with density $\pi_\delta(\theta)=\delta^{-d}\pi(\delta^{-1}\theta)$, $\theta\in\mathbb{R}^d$
and $\hat \theta_n = \hat \theta_n(X_1,\dots, X_n)$ is an arbitrary estimator based on $(X_1,\dots, X_n).$
\end{theorem}
Let us discuss two simple implications. First, if we choose $\pi=e^{-V}$ that satisfies (P1) (in view of Assumption \ref{Assump_main}) and let $\delta\to \infty,$ then Theorem~\ref{theorem:van:Trees:theta} implies
\begin{align*}
\inf_{\hat\theta_n}\sup_{\theta\in\mathbb{R}^d}\mathbb{E}_\theta\|\hat\theta_n-\theta\|^2\geq \frac{1}{n}\operatorname{tr}(\mathcal{I}^{-1}).
\end{align*}
Secondly, if we choose 
\begin{align*}
    \pi(\theta)=\prod_{j=1}^d\frac{3}{4}\cos^3(\theta_j)I_{[-\frac{\pi}{2},\frac{\pi}{2}]}(\theta_j),
\end{align*} 
that satisfies (P2), then we have, by an easy 
computation, $\mathcal{J}_\pi=\frac{9}{2}I_d.$ 
Moreover, for $\delta=\frac{2c}{\pi\sqrt{n}}$, $c>0$, the prior $\pi_\delta$ has support in $\{\theta\in\mathbb{R}^d:\max_j|\theta_j|\leq \frac{c}{\sqrt{n}}\}\subseteq\{\theta\in\mathbb{R}^d:\|\theta\|\leq c\sqrt{\frac{d}{n}}\}$, and Theorem \ref{theorem:van:Trees:theta} implies
\begin{align*}
 \inf_{\hat\theta_n}\sup_{\|\theta\|\leq c\sqrt{\frac{d}{n}}}\mathbb{E}_\theta\|\hat\theta_n-\theta\|^2\geq \frac{1}{n}\operatorname{tr}\Big(\Big(\mathcal{I}+\frac{9\pi^2}{8c^2}I_d\Big)^{-1}\Big).
\end{align*}
The proof of Theorem \ref{theorem:van:Trees:theta} will be based on the following lemma.

\begin{lemma}
\label{lem:Fisher:approxiation}
Let $\pi:\mathbb{R}^d\mapsto {\mathbb R}_+$ be a probability density function with respect to the Lebesgue measure $\lambda$ satisfying (P2) and $\int \frac{\|\pi'\|^2}{\pi}\,d\lambda<\infty$. Moreover, let $p=e^{-W}$, $W:\mathbb{R}^d\mapsto \mathbb{R}$, be a probability density function with respect $\lambda$ satisfying (P1). Suppose that $p$ is constant on the support of $\pi$ and that $\int_{\mathbb{R}^d}\|\theta\|^2\,p(\theta)d\theta<\infty$. For $\epsilon>0$, set
\begin{align}\label{eq:construction:extend:prior}
\pi_\epsilon:=\frac{q+\epsilon}{1+\epsilon}\ p,\qquad
q:=\frac{\pi}{\int \pi p\,d\lambda}.
\end{align}
Then, for every $\epsilon>0$, $\pi_\epsilon$ is a probability density function with respect to $\lambda$ satisfying (P1) and $\int\frac{\|\pi_\epsilon'\|^2}{\pi_\epsilon}\,d\lambda<\infty$. Moreover, as $\epsilon\rightarrow 0$,
\begin{align*}
\mathcal{J}_{\pi_\epsilon}=\int\frac{\pi_\epsilon'\otimes \pi_\epsilon'}{\pi_\epsilon}\,d\lambda\to \int\frac{\pi'\otimes \pi'}{\pi}\,d\lambda=\mathcal{J}_{\pi}.
\end{align*}
\end{lemma}

\begin{proof}
To see that $\pi_\epsilon$ satisfies (P1), we have to show that $W_\epsilon:\mathbb{R}^d\mapsto \mathbb{R}$ defined by $\pi_{\epsilon}=e^{-W_\epsilon}$ is twice differentiable with $\|W_\epsilon''\|_{L_{\infty}},\|W_\epsilon''\|_{\operatorname{Lip}}<\infty$. Write $q_\epsilon=\frac{q+\epsilon}{1+\epsilon}$ such that $W_\epsilon=-\log q_\epsilon+W$. Hence,
\begin{align*}
W''_\epsilon=W''-\frac{q_\epsilon''}{q_\epsilon}+\frac{q'_\epsilon\otimes q'_\epsilon}{q_\epsilon^2}=W''-\frac{q''}{q+\epsilon}+\frac{q'\otimes q'}{(q+\epsilon)^2}.
\end{align*}
Using the fact that $q+\epsilon$ is lower bounded by $\epsilon$ and that all involved functions $W'',q',q''$ and $q+\epsilon$ are bounded and have bounded Lipschitz constant, it follows that $\pi_\epsilon$ satisfies (P1). Moreover, by the assumptions on $\pi$ and $W$, we get
\begin{align*}
\pi'_\epsilon=\frac{q'e^{-W}+(q+\epsilon)W'e^{-W}}{1+\epsilon}=\frac{q'e^{-W}+\epsilon W'e^{-W}}{1+\epsilon},
\end{align*}
and
\begin{align*}
\int \frac{\|\pi_\epsilon'\|^2}{\pi_\epsilon}\,d\lambda\leq \frac{2}{1+\epsilon}\int\Big(\frac{\|q'\|^2e^{-W}}{q+\epsilon}+\|W'\|^2e^{-W}\epsilon\Big)\,d\lambda<\infty.
\end{align*}
Moreover, using again that $W$ is constant on the support of $\pi$, we get
\begin{align*}
\mathcal{J}_{\pi_\epsilon}&=\frac{1}{1+\epsilon}\int_{\mathbb{R}^d}\Big(\frac{(q'\otimes q')e^{-W}}{q+\epsilon}+\epsilon(W'\otimes W')e^{-W}\Big)\,d\lambda\\
&\to \int_{\mathbb{R}^d}\frac{(q'\otimes q')e^{-W}}{q}\,d\lambda=
\int\frac{\pi'\otimes \pi'}{\pi}\,d\lambda=\mathcal{J}_{\pi}\qquad \text{as }\epsilon\rightarrow 0,
\end{align*} 
where we applied the dominated convergence theorem in the last step.\qed
\end{proof}

We are now ready to prove Theorem \ref{theorem:van:Trees:theta}.

\begin{proof}
We first consider the case where $\pi$ satisfies (P1). We assume that the Bayes risk on the left-hand side in \eqref{eq:van:Trees:theta} is finite because otherwise the result is trivial. Since the location model is an example of an equivariant statistical model (with translation group acting on parameter space and sample space), we can apply Lemma \ref{prop:equivariant:chapman:robbins:ineq} to $\psi(\theta)=\theta$,
which yields that for any $\theta_1,\dots,\theta_d\in\mathbb{R}^d$,
\begin{align}
&\inf_{\hat\theta_n}\int_{\mathbb{R}^d}\mathbb{E}_\theta\|\hat\theta_n-\theta\|^2\,\Pi_\delta(d\theta)
\label{eq:chapman:robbins:location}
\\
&\geq \frac{\big(\sum_{j=1}^d\langle e_j,\theta_j\rangle\big)^2}{\sum_{j=1}^d\big(\chi^2(P^{\otimes n}_{\theta_j},P^{\otimes n}_0)+\chi^2(\Pi_{\delta,\theta_j},\Pi_{\delta})+\chi^2(P^{\otimes n}_{\theta_j},P^{\otimes n}_0)\chi^2(\Pi_{\delta,\theta_j},\Pi_{\delta})\big)},\nonumber
\end{align}
where $e_1,\dots,e_d$ is the standard basis in $\mathbb{R}^d$ and $\Pi_{\delta,\theta_j}$ is the probability measure with density $\delta^{-d}\pi(\delta^{-1}(\theta+\theta_j))$, $\theta\in\mathbb{R}^d$. Let us now apply a limiting argument. First, we have
\begin{align}
\chi^2(P_{\theta_j},P_0)&=\mathbb{E}e^{2V(\xi)-2V(\xi-\theta_j)}-1 \nonumber\\
&\leq e^{L\|\theta_j\|^3}\mathbb{E}e^{2\langle V'(\xi),\theta_j\rangle-\langle V''(\xi)\theta_j,\theta_j\rangle}-1,\label{eq:chi2:computation}
\end{align}
where we used Proposition \ref{remainder_bd} in the inequality. If we set $\theta_j=th_j$ with $t>0$ and $h_j\in\mathbb{R}^d$, and then combine \eqref{eq:chi2:computation} with Lemma \ref{bd_on_V'} and Assumption \ref{Assump_main}, we get
\begin{align*}
\limsup_{t\rightarrow 0}\frac{1}{t^2}\chi^2(P_{th_j},P_0)\leq -\mathbb{E}\langle V''(\xi)h_j,h_j\rangle+2\mathbb{E}\langle V'(\xi),h_j\rangle^2=\langle h_j,\mathcal{I} h_j\rangle
\end{align*}
and thus also
\begin{align*}
\limsup_{t\rightarrow 0}\frac{1}{t^2}\chi^2(P^{\otimes n}_{th_j},P^{\otimes n}_0)\leq n\langle h_j,\mathcal{I}h_j\rangle.
\end{align*}
Moreover, since the prior density $\pi$ satisfies (P1), Proposition \ref{remainder_bd} and Lemma~\ref{bd_on_V'} are still applicable, and we also have
\begin{align*}
\limsup_{t\rightarrow 0}\frac{1}{t^2}\chi^2(\Pi_{\delta,th_j},\Pi_{\delta})\leq \langle h_j,\delta^{-2}\mathcal{J}_\pi h_j\rangle.
\end{align*}
Substituting these formulas into \eqref{eq:chapman:robbins:location}, we get for every $h_1,\dots,h_d\in\mathbb{R}^d$,
\begin{align*}
 \inf_{\hat\theta_n}\int_{\mathbb{R}^d}\mathbb{E}_{\theta}\|\hat\theta_n-\theta\|^2\,\Pi_\delta(d\theta)\geq \frac{\big(\sum_{j=1}^d\langle e_j,h_j\rangle\big)^2}{\sum_{j=1}^d\langle h_j,(n\mathcal{I}+\delta^{-2}\mathcal{J}_\pi)h_j\rangle}.
\end{align*}
Setting $h_j=(n\mathcal{I}+\delta^{-2}\mathcal{J}_\pi)^{-1}e_j$, $j\leq d$, we arrive at
\begin{align}
 &\inf_{\hat\theta_n}\int_{\mathbb{R}^d}\mathbb{E}_{\delta\theta}\|\hat\theta_n-\delta\theta\|^2\,\pi(\theta)d\theta=\inf_{\hat\theta_n}\int_{\mathbb{R}^d}\mathbb{E}_\theta\|\hat\theta_n-\theta\|^2\,\Pi_\delta(d\theta)\nonumber\\
 &\geq \sum_{j=1}^d\langle e_j,(n\mathcal{I}+\delta^{-2}\mathcal{J}_\pi)^{-1}e_j\rangle=\operatorname{tr}((n\mathcal{I}+\delta^{-2}\mathcal{J}_\pi)^{-1}).\label{eq:van:Trees:location}
\end{align}

It remains to extend \eqref{eq:van:Trees:location} to all densities $\pi$ satisfying (P2). To this end, we apply Lemma \ref{lem:Fisher:approxiation} to get $\pi_\epsilon$, $\epsilon>0$, from \eqref{eq:construction:extend:prior} satisfying (P1) and $\lim_{\epsilon\rightarrow 0}\mathcal{J}_{\pi_\epsilon}=\mathcal{J}_\pi$. Note also that $qp=\pi.$ Therefore,
\begin{align*}
&
\inf_{\hat \theta_n}\int_{\mathbb{R}^d}\mathbb{E}_{\delta\theta}\|\hat\theta_n-\delta\theta\|^2\,\pi(\theta)d\theta
=
\lim_{\epsilon\to 0} \inf_{\hat \theta_n} \int_{\mathbb{R}^d}\mathbb{E}_{\delta\theta}\|\hat\theta_n-\delta\theta\|^2\, (\pi(\theta)+\epsilon p(\theta))d\theta
\\
&
= \lim_{\epsilon\to 0} 
\frac{1}{1+\epsilon}\inf_{\hat \theta_n} \int_{\mathbb{R}^d}\mathbb{E}_{\delta\theta}\|\hat\theta_n-\delta\theta\|^2\, (q(\theta)p(\theta)+\epsilon p(\theta))d\theta
\\
&
=
\lim_{\epsilon\to 0} \inf_{\hat \theta_n} \int_{\mathbb{R}^d}\mathbb{E}_{\delta\theta}\|\hat\theta_n-\delta\theta\|^2\, \pi_{\epsilon}(\theta)d\theta
\geq \lim_{\epsilon\to 0} \operatorname{tr}((n\mathcal{I}+\delta^{-2}\mathcal{J}_{\pi_\epsilon})^{-1})
\\
&
=\operatorname{tr}((n\mathcal{I}+\delta^{-2}\mathcal{J}_\pi)^{-1}),
\end{align*}
where we applied \eqref{eq:van:Trees:location} to $\pi_\epsilon.$

\qed
\end{proof}

We now turn to the estimation of functionals of the location parameter. 
For a continuous function $g:\mathbb{R}^d\rightarrow \mathbb{R}^{d}$ and $x_0\in \mathbb{R}^d$, the local continuity modulus of $g$ at point $x_0$ is defined by 
\begin{align*}
\omega_{g}(x_0,\delta)=\sup_{\|x-x_0\|\leq  \delta}\|g(x)-g(x_0)\|,\ \delta\geq 0.
\end{align*}

\begin{theorem}\label{thm:functional:lower:bound}
Let $f:\mathbb{R}^d\mapsto \mathbb{R}$ be a continuously differentiable function and let $\theta_0\in\mathbb{R}^d$. Let $\pi:\mathbb{R}\mapsto {\mathbb R}_+$ be a probability density function satisfying (P2) for $d=1$. Suppose that
\begin{align*}
\mathcal{J}_\pi=\int_{\mathbb{R}}\frac{(\pi'(s))^2}{\pi(s)}\,ds<\infty.
\end{align*}
Then there exists $v\in\mathbb{R}^d$ with $\|v\|=1$, such that, for every $\delta>0$,
\begin{align*}
 \inf_{\hat T_n}&\Big(\int_{\mathbb{R}}n\mathbb{E}_{\theta_0+sv}(\hat T_n-f(\theta_0+sv))^2\,\Pi_\delta(ds)\Big)^{1/2}
 \\
 &
 \geq \|\mathcal{I}^{-1/2}f'(\theta_0)\|-\sqrt{\frac{\mathcal{J}_\pi}{\delta^2n}}\|\mathcal{I}^{-1}f'(\theta_0)\|-\int_{\mathbb{R}}\omega_{\mathcal{I}^{-1/2}f'}(\theta_0,|s|)\,\Pi_\delta(ds),
\end{align*}
where $\Pi_\delta$ is the prior distribution with density $\delta^{-1}\pi(\delta^{-1}s)$, $s\in\mathbb{R}$ and $\hat T_n=\hat T_n(X_1,\dots, X_n)$
is an arbitrary estimator based on $(X_1,\dots, X_n).$
\end{theorem}
%
%
\begin{proof}
Without loss of generality we may assume that $\theta_0=0.$
Our goal is to apply Lemma \ref{prop:equivariant:chapman:robbins:ineq} to $\psi=f$ and to the one-dimensional subgroup $G=\mathbb{R}v=\{s v:s\in\mathbb{R}\}$ with direction $v\in\mathbb{R}^d$, $\|v\|=1$, to be determined later.
As in the proof of Theorem~\ref{theorem:van:Trees:theta}, we first establish a lower bound for densities $\pi=e^{-W}$ satisfying (P1) with $d=1$ and for the special case that $f$ and $f'$ are bounded on $\mathbb{R}v.$ 
In this case, applying Lemma \ref{prop:equivariant:chapman:robbins:ineq}, we get for every $t\in\mathbb{R}$,
\begin{align}
 &
 \inf_{\hat T_n}\int_{\mathbb{R}}\mathbb{E}_{s v}(\hat T_n-f(s v))^2\,\Pi_\delta(ds)\nonumber
 \\
 &
 \geq \frac{\big(\int_\mathbb{R}(f((s-t)v)-f(sv))\,\Pi_\delta (ds)\big)^2}{\chi^2(P^{\otimes n}_{t v},P^{\otimes n}_0)+\chi^2(\Pi_{\delta,t},\Pi_\delta)+\chi^2(P^{\otimes n}_{t v},P^{\otimes n}_0)\chi^2(\Pi_{\delta,t},\Pi_\delta)},
 \label{eq:chapman:robbins:functional}
\end{align}
where $\Pi_{\delta,t}$ is the probability measure with density $\pi_{\delta,t}(s)=\delta^{-1}\pi(\delta^{-1}(s+t))$, $s\in\mathbb{R}$. Now, using that $\pi$ satisfies (P1), we have 
\begin{align*}
&\limsup_{t\rightarrow 0}\frac{1}{t^2}\chi^2(P^{\otimes n}_{t v},P^{\otimes n}_0)\leq n \langle v, \mathcal{I}v\rangle,
\\
&\limsup_{t\rightarrow 0}\frac{1}{t^2}\chi^2(\Pi_{\delta,t},\Pi_\delta)
\leq \delta^{-2}\mathcal{J}_\pi=\delta^{-2}\int_\mathbb{R}W''(s)e^{-W(s)}\,ds,
\end{align*}
as shown in the proof of Theorem \ref{theorem:van:Trees:theta}. Moreover, using that $f$ and $f'$ are bounded on $\mathbb{R}v,$ standard results on the differentiation of integrals where the integrand depends on a real parameter (e.g.~\cite[Corollary 5.9]{MR1312157}) yield
\begin{align*}
\lim_{t\rightarrow 0}\frac{1}{t}\int_\mathbb{R}(f((s-t)v)-f(sv))\,\Pi_\delta(ds)=-\int_\mathbb{R}\langle f'(s v), v\rangle\,\Pi_\delta(ds).
\end{align*}
Substituting these formulas into \eqref{eq:chapman:robbins:location} and letting $t$ go to zero, we get
\begin{align}
\label{eq:van:Trees:functional:positive:density}
 \inf_{\hat T_n}\int_{\mathbb{R}}\mathbb{E}_{s v}(\hat T_n-f(s v))^2\,\Pi_\delta(ds)
 &\geq \frac{\big(\int_\mathbb{R}\langle f'(sv),v\rangle\,\Pi_\delta(ds)\big)^2}{n \langle v,\mathcal{I} v\rangle+\delta^{-2}\mathcal{J}_\pi}.
\end{align}
While this inequality holds for all densities $\pi$ satisfying (P1), we can apply Lemma \ref{lem:Fisher:approxiation} to extend this inequality to all probability densities satisfying (P2) (see the proof of Theorem \ref{theorem:van:Trees:theta} for the detailed argument). Moreover, for densities $\pi$ with bounded support, we can also drop the boundedness conditions on $f$ and $f'$. In fact, the latter can be achieved by applying \eqref{eq:van:Trees:functional:positive:density} to a functional $g$ with $g,g'$ bounded on $\mathbb{R}v$ and $g=f$ on the support of $\Pi_\delta$ (times $v$). As a consequence, under the assumptions of Theorem \ref{thm:functional:lower:bound}, we have for every $v\in\mathbb{R}$, 
$\|v\|=1$, 
\begin{align*}
 &\inf_{\hat T_n}\int_{\mathbb{R}}\mathbb{E}_{s v}(\hat T_n-f(s v))^2\,\Pi_\delta(ds)\\
 &\geq \frac{\big(\int_\mathbb{R}\langle f'(s v), v\rangle\,\Pi_\delta(ds)\big)^2}{n\langle v, \mathcal{I}v\rangle+\delta^{-2}\mathcal{J}_\pi}
 =\frac{1}{n}\frac{\big(\int_\mathbb{R}\langle f'(s v), v\rangle\,\Pi_\delta(ds)\big)^2}{\langle v,(\mathcal{I}+\frac{\mathcal{J}_\pi}{\delta^2n} I_d)v\rangle}.
\end{align*}
Choosing
\begin{align*}
 v=\frac{A^{-1}f'(0)}{\|A^{-1}f'(0)\|},\qquad A=\mathcal{I}+\frac{\mathcal{J}_\pi}{\delta^2n}I_d,
\end{align*}
we obtain
\begin{align*}
 \inf_{\hat T_n}\Big(\int_{\mathbb{R}}n\mathbb{E}_{s v}(\hat T_n-f(s v))^2\,\Pi_\delta(ds)\Big)^{1/2}
 &\geq\frac{|\int_\mathbb{R}\langle A^{-1/2}f'(s v),A^{-1/2}f'(0)\rangle\,\Pi_\delta(ds)|}{\|A^{-1/2}f'(0)\|}.
\end{align*}
Using the inequality
\begin{align*}
&
|\langle A^{-1/2}f'(s v),A^{-1/2}f'(0)\rangle|
\\
&
\geq \|A^{-1/2}f'(0)\|^2-\|A^{-1/2}f'(0)\|\|A^{-1/2}f'(0)-A^{-1/2}f'(s v)\|\\
&
\geq \|A^{-1/2}f'(0)\|(\|A^{-1/2}f'(0)\|-\omega_{A^{-1/2}f'}(0,|s|)),
\end{align*}
we arrive at 
\begin{align}
 &
 \inf_{\hat T_n}\Big(\int_{\mathbb{R}}n\mathbb{E}_{sv}(\hat T_n-f(s v))^2\,\Pi_\delta(ds)\Big)^{1/2}
 \nonumber
 \\
 &
 \geq \|A^{-1/2}f'(0)\|-\int_{\mathbb{R}}\omega_{A^{-1/2}f'}(0,|s|)\,d\Pi_\delta(s).\label{eq:van:Trees:positive:density:functional}
\end{align}
Since $A^{-1}=\mathcal{I}^{-1}-\frac{\mathcal{J}_\pi}{\delta^2n}A^{-1}\mathcal{I}^{-1}\succeq 0$ and 
${\mathcal I}^{-2} \succeq A^{-1}\mathcal{I}^{-1},$
we have 
\begin{align*}
\|A^{-1/2}f'(0)\|
&=\langle A^{-1}f'(0),f'(0)\rangle^{1/2}
\\
&
=(\langle {\mathcal I}^{-1}f'(0),f'(0)\rangle - \langle\frac{\mathcal{J}_\pi}{\delta^2n}A^{-1}\mathcal{I}^{-1} f'(0), f'(0)\rangle)^{1/2}
\\
&
\geq \langle {\mathcal I}^{-1}f'(0),f'(0)\rangle^{1/2}-\langle\frac{\mathcal{J}_\pi}{\delta^2n}A^{-1}\mathcal{I}^{-1} f'(0),f'(0)\rangle^{1/2}
\\
&\geq \|\mathcal{I}^{-1/2}f'(0)\|-\sqrt{\frac{\mathcal{J}_\pi}{\delta^2n}}\|\mathcal{I}^{-1}f'(0)\|.
\end{align*}
Note also that $A^{-1} \preceq {\mathcal I}^{-1},$ implying $\|A^{-1}u\|\leq \|{\mathcal I}^{-1}u\|, u\in {\mathbb R}^d$
and, as a consequence, $\omega_{A^{-1/2}f'}(0,|s|)\leq \omega_{\mathcal{I}^{-1/2}f'}(0,|s|).$
These bounds along with \eqref{eq:van:Trees:positive:density:functional} imply the claim of the theorem.
\qed
\end{proof}
  
Finally, we prove Proposition \ref{min_max_lower_local}.

\begin{proof} Let us choose $\pi(s)=\frac{3}{4}\cos^3(\theta)I_{[-\frac{\pi}{2},\frac{\pi}{2}]}(\theta)$ in which case (P2) holds and $\mathcal{J}_\pi=\frac{9}{2}$. 
Choosing additionally $\delta=\frac{2c}{\pi\sqrt{n}}$, $c>0$, Theorem \ref{thm:functional:lower:bound} yields
\begin{align}
 \label{bd_on_inf_sup}
 &
 \nonumber
 \inf_{\hat T_n}\sup_{\|\theta-\theta_0\|\leq \frac{c}{\sqrt{n}}}(n\mathbb{E}_\theta(\hat T_n-f(\theta))^2)^{1/2}
\\
&
 \geq \|\mathcal{I}^{-1/2}f'(\theta_0)\|-\frac{3\pi}{\sqrt{8}c}\|\mathcal{I}^{-1}f'(\theta_0)\|
 -\omega_{\mathcal{I}^{-1/2}f'}\Big(\theta_0,\frac{c}{\sqrt{n}}\Big).
\end{align}
Under Assumption \ref{Assump_main}, $\|\mathcal{I}^{-1}f'(\theta_0)\|\leq \frac{1}{\sqrt{m}}\|\mathcal{I}^{-1/2}f'(\theta_0)\|.$
In addition, for $f\in C^{s},$ where $s=1+\rho,$ $\rho\in (0,1],$ 
\begin{align*}
\omega_{\mathcal{I}^{-1/2}f'}\Big(\theta_0,\frac{c}{\sqrt{n}}\Big)\leq 
\frac{1}{\sqrt{m}}\omega_{f'}\Big(\theta_0,\frac{c}{\sqrt{n}}\Big)\leq 
\frac{1}{\sqrt{m}}\|f\|_{C^s} \Bigl(\frac{c}{\sqrt{n}}\Bigr)^{\rho}.
\end{align*}
Recalling that $\|\mathcal{I}^{-1/2}f'(\theta_0)\|=\sigma_f(\theta_0),$ bound \eqref{bd_on_inf_sup} implies 
\begin{align}
\label{bd_on_inf_sup_AA}
&
\nonumber
\inf_{\hat T_n}\sup_{\|\theta-\theta_0\|\leq \frac{c}{\sqrt{n}}}\frac{\sqrt{n}\|\hat T_n-f(\theta)\|_{L_2({\mathbb P}_{\theta})}}{\sigma_f(\theta_0)}
\\
&
\geq 1- \frac{3\pi}{\sqrt{8m}c}
 -\frac{1}{\sqrt{m}}\frac{\|f\|_{C^s}}{\sigma_f(\theta_0)} 
 \Bigl(\frac{c}{\sqrt{n}}\Bigr)^{\rho}.
\end{align}
Note that, for all $\theta$ satisfying $\|\theta-\theta_0\|\leq \frac{c}{\sqrt{n}},$
\begin{align*}
&
|\sigma_f(\theta)-\sigma_f(\theta_0)|= |\|\mathcal{I}^{-1/2}f'(\theta)\|-\|\mathcal{I}^{-1/2}f'(\theta_0)\|| 
\\
&
\leq \omega_{\mathcal{I}^{-1/2}f'}\Big(\theta_0,\frac{c}{\sqrt{n}}\Big)
\leq \frac{1}{\sqrt{m}}\|f\|_{C^s} \Bigl(\frac{c}{\sqrt{n}}\Bigr)^{\rho}.
\end{align*}
Therefore,
\begin{align*}
&
\sup_{\|\theta-\theta_0\|\leq \frac{c}{\sqrt{n}}}\frac{\sqrt{n}\|\hat T_n-f(\theta)\|_{L_2({\mathbb P}_{\theta})}}{\sigma_f(\theta_0)}
\\
&
\leq 
\sup_{\|\theta-\theta_0\|\leq \frac{c}{\sqrt{n}}}\frac{\sqrt{n}\|\hat T_n-f(\theta)\|_{L_2({\mathbb P}_{\theta})}}{\sigma_f(\theta)}
\sup_{\|\theta-\theta_0\|\leq \frac{c}{\sqrt{n}}}\frac{\sigma_f(\theta)}{\sigma_f(\theta_0)}
\\
&
\leq 
\sup_{\|\theta-\theta_0\|\leq \frac{c}{\sqrt{n}}}\frac{\sqrt{n}\|\hat T_n-f(\theta)\|_{L_2({\mathbb P}_{\theta})}}{\sigma_f(\theta)}
\Bigl(1+\sup_{\|\theta-\theta_0\|\leq \frac{c}{\sqrt{n}}}\frac{|\sigma_f(\theta)-\sigma_f(\theta_0)|}{\sigma_f(\theta_0)}\Bigr)
\\
&
\leq \sup_{\|\theta-\theta_0\|\leq \frac{c}{\sqrt{n}}}\frac{\sqrt{n}\|\hat T_n-f(\theta)\|_{L_2({\mathbb P}_{\theta})}}{\sigma_f(\theta)}
\Bigl(1+ \frac{1}{\sqrt{m}}\frac{\|f\|_{C^s}}{\sigma_f(\theta_0)} \Bigl(\frac{c}{\sqrt{n}}\Bigr)^{\rho}\Bigr).
\end{align*}
Using this bound together with \eqref{bd_on_inf_sup_AA} easily yields
\begin{align*}
&
\sup_{\|\theta-\theta_0\|\leq \frac{c}{\sqrt{n}}}\frac{\sqrt{n}\|\hat T_n-f(\theta)\|_{L_2({\mathbb P}_{\theta})}}{\sigma_f(\theta)}
\geq 1- \frac{3\pi}{\sqrt{8m}c}
 -\frac{2}{\sqrt{m}}\frac{\|f\|_{C^s}}{\sigma_f(\theta_0)} 
 \Bigl(\frac{c}{\sqrt{n}}\Bigr)^{\rho}.
 \end{align*}

\qed
\end{proof}


\end{document}